\documentclass[10pt]{amsart}
\usepackage{amsfonts,amsmath,amssymb,amscd,latexsym,epsfig,color,a4wide,mathabx,epstopdf}
\usepackage{graphicx}

\newtheorem{theorem}{Theorem}[section]    % Standard theorem environment
\newtheorem{lemma}[theorem]{Lemma}          % Lemma environment with numbering 
\newtheorem{proposition}[theorem]{Proposition}
\newtheorem{prop-defn}[theorem]{Proposition-Definition}  
\newtheorem{claim}[theorem]{Claim}  
\newtheorem{corollary}[theorem]{Corollary} 

\theoremstyle{definition}
\newtheorem{definition}[theorem]{Definition}
\newtheorem{remark}[theorem]{Remark}

\newtheorem{example}[theorem]{Example}

 \makeatletter
    
    \@addtoreset{equation}{section}
  \makeatother
\def\co{\colon \thinspace}

\newcommand{\GL}{\textrm{GL}}
\newcommand{\Z}{\mathbb{Z}}

\newcommand{\C}{\mathbb{C}}

\newcommand{\LWcr}{\textsf{LWcr}}
\newcommand{\SWcr}{\textsf{SWcr}}
\newcommand{\Wcr}{\textsf{Wcr}}
\title[Garside theory and Burau representation]{Garside-theoretic 
analysis of Burau representations}
\author{Matthieu Calvez}
\author{Tetsuya Ito}

\address{Matthieu Calvez, Departamento de Matem\'{a}tica y Ciencia de la Computaci\'{o}n, Facultad de Ciencias, Universidad de Santiago de Chile, Avenida Libertador Bernardo O'Higgins~3363, Santiago, Chile}
\email{calvez.matthieu@gmail.com}

\address{Research Institute for Mathematical Sciences, Kyoto university
Kyoto, 606-8502, Japan}
\email{tetitoh@kurims.u-kyoto.ac.jp}
\urladdr{http://kurims.kyoto-u.ac.jp/~tetitoh/}


\begin{document}

%\marg{M:I feel we could improve the title but I do not have better idea right now, abstract adapted}
\begin{abstract} 
We establish relations between both the classical and the dual Garside structures of the braid group and the Burau representation. 
Using the classical structure, 
we formulate a non-vanishing criterion for 
%restrictions for braids in the kernel of 
the Burau representation of the 4-strand braid group. In the dual context, it is shown that the Burau representation for arbitrary braid index is injective when restricted to the set of \emph{simply-nested braids}.   
\end{abstract}
\maketitle

\section{Introduction}\label{section:introduction}

The (reduced) \emph{Burau representation} 
$$\rho_n:\ \ B_n \longrightarrow \text{GL}\left(n-1,\mathbb Z[q^{\pm 1}]\right)$$ 
was the first possible candidate for a faithful linear representation of the braid group on $n$ strands $B_n$ and it has been known for long to be faithful in the case of the 3-strand braid group~\cite{magnuspeluso}. 
However, Moody \cite{moody} showed that the Burau representation is not faithful for any braid index $n\geqslant 9$. This was brought down to $n\geqslant 6$ by Long and Paton \cite{lp} and finally Bigelow showed the non-faithfulness of $\rho_5$ \cite{bigelow}. Despite these negative results, the linearity question of the braid groups was settled in the positive independently by Krammer \cite{krammer} and Bigelow \cite{bigelowlinear}. They showed that another linear representation 
$$\mathcal L_n:\ \ B_n\longrightarrow GL\left(\frac{n(n-1)}{2},\mathbb Z[q^{\pm 1},t^{\pm 1}]\right)$$ 
constructed by Lawrence \cite{lawrence} is faithful for all $n$. The representation $\mathcal{L}_{n}$ is now known as the \emph{Lawrence-Krammer-Bigelow representation}, or LKB representation for short.

At present, the question of the faithfulness of the Burau representation in the case $n=4$ remains open. 
The linearity question itself was solved, nevertheless the problem to determine whether $\rho_4$ is faithful or not remains of considerable importance:
a negative answer would be of great interest in 
quantum topology, since it is equivalent to the non-faithfulness of Jones and Temperley-Lieb representations of $B_{4}$, and would %``almost certainly" 
provide a non-trivial knot with trivial Jones polynomial \cite{bi3}.
Another interesting related problem is to study the image of the Burau representation -- an old question asks which $(n-1)\times (n-1)$ matrices over $\Z[q^{\pm 1}]$ can appear as the image under the (reduced) Burau representation of some braid \cite{birman}; it is widely open.

The present paper aims to establish relations between the Garside structures of the braid group and the Burau representation. Our motivation was to understand to what extent the Burau representation is close to be faithful, and when faithfulness property breaks down. This not only helps to attack the faithfulness problem of the 4-strand Burau representation, but also provides new insights for the image and the kernel of the Burau representation for arbitrary braid index, even for the simplest case $n=3$ (see Corollary \ref{cor:3strands} below).

The \emph{classical Garside structure}
consists in a lattice structure together with a special element $\Delta$ satisfying some properties initially discovered by Garside in \cite{garside}. 
A crucial output of this structure is the
\emph{classical (left) normal form} of a braid $x$, which is a unique decomposition of the form 
$$N_{\sf c}(x) = \Delta^p s_1\cdots s_r$$
 in which the factors belong to the set of the so-called \emph{simple elements}. 
The \emph{classical} \emph{supremum} and \emph{infimum} of $x$ are defined by $\sup_{\sf c}(x)= p+r$, $\inf_{\sf c}(x)=p$, respectively. The \emph{classical canonical length} of $x$ is defined by $\ell_{\sf c}(x)=r$; 
the \emph{classical Garside length} $l_{\sf c}(x)$ is the length of $x$ with respect to the simple elements. The latter satisfies
 $l_{\sf c}(x)=\max(\sup_{\sf c}(x),0)-\min(\inf_{\sf c}(x),0)$.
%\marg{temptative notation: $\sup_{\sf c}$, $\inf_{\sf d}$ etcetera for classical, dual sup, inf,...
%(Ito: OK, let us use these notations. )}

Slightly different but very close in spirit is the \emph{dual Garside structure} (or BKL structure) discovered by Birman, Ko and Lee \cite{bkldual} which leads to the  \emph{dual (left) normal form} of a braid $x$:
$$N_{\sf d}(x) = \delta^p d_1\cdots d_r$$
where the factors belong to the set of the so-called \emph{dual simple elements}.
The dual \emph{supremum}, \emph{infimum} and \emph{canonical length} of a braid $x$ are defined similarly and denoted by $\sup_{\sf d}(x)$, $\inf_{\sf d}(x)$ and $\ell_{\sf d}(x)$ respectively. The \emph{dual Garside length} $l_{\sf d}(x)$ is the length of $x$ with respect to the dual simple elements; it satisfies $l_{\sf d}(x)=\max(\sup_{\sf d}(x),0)-\min(\inf_{\sf d}(x),0)$. 
See Section \ref{section:garside} for more details on both classical and dual Garside structures of the braid group.

Our first main result provides a non-vanishing criterion for the Burau representation $\rho_4$ using the classical Garside structure.
\theoremstyle{plain}
\newtheorem*{thmMainclassical}{Theorem \ref{theorem:main_classical}}
\begin{thmMainclassical}
If the classical left normal form of a 4-braid $x$ %with even infimum 
does not contain a factor $(\sigma_{2}\sigma_{1}\sigma_{3})$ %and starts with a factor $x_1$ with $S(x_1)\neq \{1,3\}$ 
then $\rho_{4}(\beta) \neq 1$.   
\end{thmMainclassical}

In the dual framework, we obtain more general and strong connections.
For a non-zero Laurent polynomial $\Lambda$ in the variable $q$, 
let us denote by $m(\Lambda)$ and $M(\Lambda)$ the minimal and maximal degrees of the variable, respectively. As a convention, we define $m(0)= +\infty$ and $M(0)=-\infty$.

 For a matrix $\Lambda=(\Lambda_{ij})\in GL\left(n-1,\mathbb Z[q^{\pm 1}]\right)$, 
we set $$m(\Lambda)=\min \{m(\Lambda_{ij}), 1\leqslant i,j\leqslant n-1\},\ \textrm{and}  \ M(\Lambda)=\max \{M(\Lambda_{ij}, 1\leqslant i,j\leqslant n-1\}.$$

%\marg{Changed -- since statement of theorem was changed}
In Section \ref{section:dual} we will introduce a notion of \emph{simply-nested} braid; roughly speaking, simply-nestedness is a local condition 
on the factors of the dual left normal form of a braid. 
We will show that the Burau representation completely determines the normal form of simply-nested braids.
%We will show that the Burau representation detects the dual Garside length of simply-nested braids, like the LKB representation does for any braid \cite{iw}. 

%\marg{%There was no confusion but I had the feeling that canonical length is more common when speaking of Garside theory. I propose to use $\ell$ as usual for the can. length and $l$ for Garside length. 
%Anyway the (3) part of the statements is trivial when one knows (1) and (2), maybe we could remove...??
%(Ito) OK. I remove the statements for length.}

\newtheorem*{ThmMainDual}{Theorem \ref{theorem:main_dual}}
\begin{ThmMainDual}
Let $x \in B_n$ be a simply-nested braid. 
\begin{enumerate}
\item[(i)] $\sup_{\sf d}(x) = M(\rho_n(x))$.
%\item[(ii)] $\inf_{\sf d}(x) = m(\rho_n(x))$
\item[(ii)] One can compute the dual normal form from the matrix $\rho_{n}(x)$, so the restriction of the Burau representation on the set of simply-nested braid $B_{n}^{\sf sn}$ is injective.
\end{enumerate} \end{ThmMainDual}

This provides several consequences for faithfulness questions in general.
First of all, it follows that the Burau matrix of a 3-strand braid completely determines its dual normal form.
\newtheorem*{Corollary3strds}{Corollary \ref{cor:3strands}}
\begin{Corollary3strds}
Let $x\in B_3$. Then 
\begin{enumerate}
\item[(i)] $\sup_{\sf d}(x) = M(\rho_{3}(x))$,
\item[(ii)] $\inf_{\sf d}(x) = m(\rho_{3}(x))$.
\item[(iii)] One can compute the dual normal form of $x$ from the matrix $\rho_{3}(x)$.
\end{enumerate}\end{Corollary3strds}

%For the 4-braid group, there are only two non-simply nested normal form of length two: $a_{2,4}(a_{1,2}a_{3,4})$ and $a_{1,3}(a_{2,3}a_{1,4})$. 

For the 4-strand braid group, we will see:

%\marg{Added a statement similar to Theorem 3.4}
\newtheorem*{Corollary4stds}{Corollary \ref{cor:4strands}}
\begin{Corollary4stds}
Let $x\in B_4$ and $N_{\sf d}(x)=\delta^pd_1\cdots d_r$. Assume that for all $i=1,\ldots,r-1$,
$(d_i,d_{i+1})$ is not in the following list:
$$ \left\{ \begin{array}{c}(a_{1,2}a_{3,4},a_{2,4}),(a_{1,2}a_{3,4},a_{3,4}a_{2,3}),(a_{1,2}a_{3,4},a_{1,2}a_{1,4}),\\
\ (a_{2,3}a_{1,4},a_{1,3}), (a_{2,3}a_{1,4},a_{1,3}a_{2,3}),(a_{2,3}a_{1,4},a_{1,3}a_{1,4}) \end{array}\right\}$$
%
%\begin{tabular}{cc}
%$\{(a_{1,2}a_{3,4},a_{2,4}),(a_{1,2}a_{3,4},a_{3,4}a_{2,3}),(a_{1,2}a_{3,4},a_{1,2}a_{1,4}),$ & \\
% & $(a_{2,3}a_{1,4},a_{1,3}), (a_{2,3}a_{1,4},a_{1,3}a_{2,3}),(a_{2,3}a_{1,4},a_{1,3}a_{1,4})\}.$
% \end{tabular}
%$\{(a_{1,2}a_{3,4},a_{2,4}),(a_{1,2}a_{3,4},a_{3,4}a_{2,3}),(a_{1,2}a_{3,4},a_{1,2}a_{1,4}),$ 
%\begin{flushright}{$(a_{2,3}a_{1,4},a_{1,3}), (a_{2,3}a_{1,4},a_{1,3}a_{2,3}),(a_{2,3}a_{1,4},a_{1,3}a_{1,4})\}.$}
 % \end{flushright}
Then 
\begin{enumerate}
%\item $\inf_{\sf d}(x) = m_{q}(\rho_{4}(x))$,
\item[(i)] $\sup_{\sf d}(x) = M_{q}(\rho_{4}(x))$,
\item[(ii)] one can compute the dual normal form of $x$ from the matrix $\rho_4(x)$. 
\end{enumerate}
In particular, if the dual left normal form of a 4-braid $x$ 
does not contain a factor $(a_{1,2}a_{3,4})$ or $(a_{2,3}a_{1,4})$ then $\rho_{4}(\beta) \neq 1$.
\end{Corollary4stds}

Finally we give Garside-theoretical constraints for braids of arbitrary braid index to belong to the kernel of the Burau representation. 
Let $e:B_n \rightarrow \mathbb Z$ be the abelianization map.  
\newtheorem*{Corollarynstds}{Corollary \ref{cor:nstrands}}
\begin{Corollarynstds}
Let $x\in B_{n}$ be a non-trivial braid and $N_{\sf d}(x)=\delta^{p}d_{1}\cdots d_{r}$. If there exists $r'\leqslant r$ such that
\begin{enumerate}
\item[(i)] The subword $x_{r'}=\delta^{p}d_{1}\cdots d_{r'}$ is simply-nested,
\item[(ii)] $r' > e(d_{r'+1}\cdots d_{r})$,
\end{enumerate}
then $\rho_n(x)\neq 1$.\end{Corollarynstds}

Thus, we conclude that if a braid $x$ is sufficiently close to simply nested braids, then its Burau matrix is never trivial.

Now we explain the organization of the paper.  
In section \ref{section:garside} we recall Garside theoretical no(ta)tions to be used later. 
Section \ref{section:classical} shows Theorem \ref{theorem:main_classical}.
Sections \ref{section:diagrams}-\ref{section:dual} are devoted to the proof of Theorem \ref{theorem:main_dual}. This 
can be sketched as follows. First we recall from \cite{iw} the wall-crossing labeling of the curve diagram of a braid and how it is related to the dual Garside normal form (Section \ref{section:diagrams}). Section \ref{section:burau} reviews a homological interpretation of the reduced Burau representation; in this context we show how the Burau matrix is related to the wall-crossing labeling.
Wall-crossing labeling therefore serves as a bridge between Burau representation and the dual Garside structure. 
Finally, Section \ref{section:dual} introduces the notion of simply-nestedness and proves Theorem \ref{theorem:main_dual}
and its above mentionned corollaries.

%*********************************************************************************

\section{Reminders on the Garside structures of braid groups}\label{section:garside}

Let $\mathbb D^2$ be the closed disk in $\mathbb C$ with diameter the real segment $[0,n+1]$ and $\mathbb D_n$ be the $n$-times punctured disk: 
$\mathbb D_{n}=\mathbb D^2-\{1,\ldots,n \}$. 
We denote the $i$-th puncture point $i \in \C$ by $p_{i}$ and put $p_0=0\in \mathbb C$.
As is well-known, the braid group $B_{n}$ is identified with the mapping class group of $\mathbb D_{n}$ (with boundary fixed pointwise).
We identify the standard Artin generator $\sigma_{i}$ ($i=1,\ldots,n-1$) with the \emph{left-handed} (that is, \emph{clockwise}) half Dehn twist along the real segment $[i,i+1]$. Throughout the paper we will consider braids acting on the \emph{right}.

%Of particular importance is the following family of elements in $B_n$: 
For $1\leqslant i\neq j\leqslant n$, we denote by $a_{i,j}$ (or $a_{j,i}$ indifferently) the isotopy class of the left-handed half Dehn twist along an arc connecting the punctures $p_i$ and $p_j$ through the lower part of the disk $\{z \in \mathbb D^{2}\: | \: \textrm{Im}\, z < 0\}$. Using the Artin generators, $a_{i,j}$ $(i<j)$ can be written as 
$$ a_{i,j} = (\sigma_{j-2}\cdots \sigma_{i+1}\sigma_{i})^{-1} \sigma_{j-1}(\sigma_{j-2}\cdots\sigma_{i+1}\sigma_{i}).$$
%The usual Artin generators $\sigma_i$ are recovered as elements of that family: for $i=1\ldots, n-1$, $\sigma_i=a_{i,i+1}$. 

\subsection{The classical Garside structure}\label{section:garsideclassical}

Let $B_{n}^{+}$ be the monoid of \emph{positive braids}, i.e. those braids which can be expressed as words on the letters $\sigma_i$ with only positive exponents.
Since the works of Garside \cite{garside}, Thurston \cite{echlpt}, and ElRIfai and Morton \cite{em}, it is well-known that the monoid $B_{n}^{+}$ induces a lattice order $\preccurlyeq_{\sf c}$ on $B_n$ called the \emph{prefix} order, through the relation $x\preccurlyeq_{\sf c} y$ if and only if $x^{-1}y \in B_n^{+}$. 

The positive left-divisors (with respect to $\preccurlyeq_{\sf c}$) of the half-twist of all strands 
$$\Delta = (\sigma_{1}\sigma_{2}\cdots\sigma_{n-1})\cdots(\sigma_{1}\sigma_{2})(\sigma_{1})$$ 
% whose positive left-divisors (with respect to $\preccurlyeq$)
are finitely many and generate the group $B_n$. These are called \emph{simple elements} (or \emph{positive permutation braids}, because they are in one-to-one correspondence with the symmetric group on $n$ objects). Simple elements have been extensively studied in \cite{em}, where a nice geometric description of them is given: a positive braid $x$ is a simple element if and only if each pair of strands in $x$ has at most one crossing. 
The pair $(B_n^{+},\Delta)$ is generally called \emph{classical (usual) Garside structure} of the braid group.

An ordered pair of two simple elements $(s ,s')$ is said to be \emph{left-weighted} if $\Delta \wedge_{\sf c}(ss') = s$, where $\wedge_{\sf c}$ denotes the greatest common divisor with respect to the lattice ordering $\preccurlyeq_{\sf c}$. 

\begin{prop-defn}\cite{em}
Let $x\in B_n$. There exists a unique decomposition of $x$ of the form 
$$N_{\sf c}(x)=\Delta^{p}s_1\cdots s_r,$$
where $p\in \mathbb Z$ and $s_1,\ldots, s_r$ are simple elements with $s_1\neq \Delta, s_r\neq 1$ such that (provided $r\geqslant 2$) for each $i=1,\ldots, r-1$, the pair $(s_{i},s_{i+1})$ is left-weighted. We call $N_{\sf c}(x)$ the classical (left) normal form of $x$.

\end{prop-defn}

The notion of left-weightedness for the usual Garside structure is understood as follows. The \emph{starting set} and the \emph{finishing set} of a simple element $s$ are defined by 
$$S(s)=\{i\in \{1,\ldots,n-1\}\: | \:  \sigma_{i}^{-1} s \in B_{n}^{+}\},$$
$$F(s)=\{i\in \{1,\ldots, n-1\}\: | \: s\sigma_{i}^{-1} \in B_{n}^{+}\},$$
respectively. In terms of crossings of braid diagrams, $i\in S(x)$ ($i\in F(x)$, respectively) if and only if the strands numbered $i$ and $i+1$ at the beginning of $x$ (at the end of $x$, respectively) do cross in $x$. 

\begin{proposition}\cite{em} 
An ordered pair of two simple elements $(s,s')$ is left-weighted if and only if $S(s')\subset F(s)$.
\end{proposition}

Thus, in terms of crossings the left-weightedness condition says that no crossing $\sigma$ from $s'$ can be moved to $s$ in such a way that $s\sigma$ is still simple.

\subsection{The dual Garside structure}\label{section:garsidedual}

Let $B_{n}^{+\ast}$ be the \emph{monoid of dual positive braids}, generated by positive powers of all elements in the family $\{a_{i,j}\}_{1\leqslant i<j\leqslant n}$ and  $\delta=\sigma_{n-1}\cdots\sigma_{2}\sigma_{1}$ be the braid corresponding to the clockwise rotation of all strands by one notch.

Birman, Ko and Lee \cite{bkldual} showed that $(B_n^{+\ast},\delta)$ is another Garside structure for the braid group. In particular the monoid $B_n^{+\ast}$ induces a lattice order on $B_n$, which we denote by $\preccurlyeq_{\sf d}$ ($x\preccurlyeq_{\sf d} y \Leftrightarrow x^{-1}y\in B_n^{+\ast}$) and % with respect to which 
the dual positive divisors of~$\delta$ (with respect to $\preccurlyeq_{\sf d}$) form a finite generating set % of the braid group,
called the set of \emph{dual simple elements}. The pair $(B_n^{+\ast},\delta)$ is called the \emph{dual Garside structure} of the braid group.

 The notion of left-weightedness is defined in the same way as in the classical case: an ordered pair of dual simple elements $(d,d')$ is left-weighted if $\delta \wedge_{\sf d} (dd') = 
d$, here $\wedge_{\sf d}$ is the greatest common divisor with respect to $\preccurlyeq_{\sf d}$. Then we have, analogous to the classical left normal form, the \emph{dual left normal form}.

\begin{prop-defn}\cite{bkldual}
Let $x\in B_n$. There exists a unique decomposition of $x$ of the form 
$$ N_{\sf d}(x)=\delta^{p} d_1\cdots d_r,$$
 where $p\in \mathbb Z$ and $d_1,\ldots,d_r$ are dual simple elements with $d_{1} \neq \delta$, $d_r\neq 1$ such that (provided $r\geqslant 2$) for each $i=1,\ldots, r-1$, the pair $(d_i,d_{i+1})$ is left-weighted. We call $N_{\sf d}(x)$ the dual (left) normal form of $x$.
\end{prop-defn}

The dual simple elements can be more easily described and studied viewing them as mapping classes of the punctured disk $\mathbb D_n$. To this end we isotope the latter to the following model:
%$\{z\in \mathbb C, |z|\leqslant 1\}-\{p_k=\frac{1}{2}e^{i\pi-\frac{2ik\pi}{n}},k=1,\ldots,n\}$; 
%$\{z\in \mathbb C, |z|\leqslant 1\}-\{p_k=\frac{1}{2}e^{i\frac{\pi}{n}(n+1-2k)},k=1,\ldots,n\}$; \marg{Modified definition of $\mathbb D_n$ to avoid fractions}
$$\{z\in \mathbb C, |z|\leqslant 2\}-\{p_i= e^{\sqrt{-1}\frac{\pi}{n}(n+1-2i)},i=1,\ldots,n\};$$
we denote by $\Gamma$ the circle $|z|=1$ along which the punctures are placed. For simplicity the $i$th puncture will be denoted by $i$ instead of $p_i$ and for $i,j\in \{1,\ldots,n\}$, we also denote by $(i,j)$ the arc of $\Gamma$ described by the move of the puncture $i$ clockwise along $\Gamma$ until the position $j$. The generator $a_{i,j}$ is then a clockwise (left-handed) half Dehn twist along the chord segment connecting the punctures $i$ and $j$.

Let us now describe the set of dual simple elements. For $r=2,\ldots,n$, take $r$ punctures $i_1,\ldots,i_r$ in this order when running along $\Gamma$ clockwise from $i_1$ to $i_r$. All braid words obtained as a concatenation of $r-1$ consecutive letters taken from the sequence 
$(a_{i_r,i_1},a_{i_{r-1},i_{r}},\ldots,a_{i_1,i_2})$ in this order, up to cyclic permutation, represent the same braid $P$.

Geometrically, as a mapping class of $\mathbb D_n$, the braid $P$ corresponds to a clockwise rotation by one notch of a neighborhood of the convex polygon in $\mathbb D_n$ whose vertices are the punctures $i_1,\ldots,i_r$. Due to this correspondence,  we call such a braid $P$ a \emph{(convex) polygon}
and we will often confuse $P$ with the corresponding convex polygon in $\mathbb D_n$. 
For example, the dual Garside element $\delta$ corresponds to the polygon which is the convex hull of all punctures.
Notice that when $r=2$, the polygon is degenerated and corresponds to a single letter $a_{i_1,i_2}$ (a single half Dehn twist about the chord segment joining the two punctures). 
Two polygons which are disjoint commute; their respective actions on the disk are totally independent from each other. Any dual simple element can be written in a unique manner as a product of disjoint polygons (up to permutation of the factors) \cite{bkldual}. 

The notion of left-weightedness in the dual context can be described as follows. Let $a_{i,j}$ and $a_{k,l}$ be two generators. 
We say that $a_{k,l}$ \emph{obstructs} $a_{i,j}$ and we write $a_{k,l}\vdash a_{i,j}$ if $k\in (j,i-1)$ and $l\in (i,j-1)$.
The relation $\vdash$ is not symmetric: $a_{k,l}\vdash a_{i,j}$ does not imply $a_{i,j}\vdash a_{k,l}$.

\begin{proposition}
\label{prop:left-weighted}\cite{bkldual}
Let $d=P_{1}\cdots P_{r}$ and $d'=Q_1\cdots Q_{r'}$ be dual simple elements expressed as products of disjoint polygons. 
Then the pair $(d,d')$ is left-weighted if and only if for any two vertices $i$ and $j$ of a polygon among $Q_1,\ldots,Q_{r'}$, there exists a polygon among $P_1,\ldots,P_{r}$ having two vertices $k$ and $l$ such that $a_{k,l}\vdash a_{i,j}$. 
%$k\in (j,i-1)$ and $l\in (i,j-1)$. 
\end{proposition}

\section{Burau representation and the classical Garside structure of $B_4$}\label{section:classical}

This section originated in trying to exploit a result by Lee and Song which can be stated as follows: 

\begin{theorem}\cite{leesong}
If non-trivial, the kernel of the Burau representation $\rho_4$ is a pseudo-Anosov subgroup of $B_4$. 
\end{theorem}

Pseudo-Anosov braids are mapping classes of the punctured disk $\mathbb D_n$ represented by pseudo-Anosov homeomorphisms, those which are neither a root of the full twist $\Delta^2$, nor permute a family of disjoint isotopy classes of simple closed curves in $\mathbb D_n$ \cite{farbmargalit}. 

An important result relating pseudo-Anosov braids and Garside theory asserts that any pseudo-Anosov braid admits a power which is conjugate to a rigid braid \cite{bggm1}, meaning that it is cyclically left-weighted: the ordered pair formed by the last and the first factor is left-weighted. 
Moreover, up to taking further power we may assume this rigid braid to have even infimum.

Observe now that the Burau matrix $\rho_4(\Delta^2)$ is the homothety of ratio $q^4$: $\rho_4(\Delta^2)=q^4I_3$. 
It follows that the Burau representation $\rho_4$ is not faithful if and only if there exists a rigid pseudo-Anosov positive braid
with infimum 0 whose Burau matrix is an homothety of ratio $q^{4p}$ for some positive integer $p$.
This motivates to explore some conditions under which the Burau matrix is not an homothety.

Let $x\in B_4$. For $i=1,2,3$ we define 
$M_i(x)=\max \{M(\rho_4(x)_{ij}), 1\leqslant j\leqslant 3\},$ in words the maximal degree of the variable $q$ among the Laurent polynomials
apparing in the $i$th row of the reduced Burau matrix of $x$.  

%{\color{red} Let $\rho_{4}(x) =(a_{i,j})\in GL\left(n-1,\mathbb Z[q]\right)$ be the reduced Burau Matrix of a 4-braid $x$.
%For $i=1,2,3$ we define  
%%For a matrix $A=(a_{i,j})\in GL\left(n-1,\mathbb Z[q]\right)$ we define, for $i=1,\ldots,3$:
%$$M_i(x)=\max \{M(a_{i,j}), 1\leqslant j\leqslant n-1\}.$$
%}
We recall the following computations (see Section \ref{section:burau}):
$$\rho_4(\sigma_1)=
\begin{pmatrix}
-q & 0 & 0\\
1 & 1 & 0\\
0 & 0 & 1\\
\end{pmatrix};
\rho_4(\sigma_2)=
\begin{pmatrix}
1 & q & 0\\
0 & -q & 0\\
0 & 1 & 1\\
\end{pmatrix};
\rho_4(\sigma_{3}) =
\begin{pmatrix}
1 &  0 & 0 \\
0 & 1 & q \\
0 & 0  & -q \\
\end{pmatrix}.
$$

\begin{lemma}\label{lem:classicalLemme}
Let $x\in B_4$. Suppose $x$ has infimum 0 and $s_1\cdots s_r$ ($r\geqslant 2$) is the classical normal form of $x$. Suppose that for all $i=1,\ldots,r$, $s_i\neq \sigma_2\sigma_1\sigma_3$. 
Denote simply by $M_i$ the integer $M_i(x)$. 
Then we have the following: 
\begin{itemize}
\item If $S(s_1)=\{1\}$ then $M_1>M_2$ and $M_1>M_3+1$, 
\item if $S(s_1)=\{2\}$ then $M_2\geqslant M_1$ and $M_2>M_3$, 
\item if $S(s_1)=\{3\}$ then $M_3\geqslant M_1$ and $M_3\geqslant M_2$, 
\item if $S(s_1)=\{1,2\}$ then either $M_1>M_2$ and $M_1>M_3+1$, or $M_2\geqslant M_1$ and $M_2>M_3$,
\item if $S(s_1)=\{2,3\}$ then either $M_2\geqslant M_1$ and $M_2>M_3$, or $M_3\geqslant M_1$ and $M_3\geqslant M_2$, 
\item if $S(s_1)=\{1,3\}$ then either $M_1>M_2$ and $M_1>M_3$,
%$\mathbf{M_1>M_3}$, 
or $M_3\geqslant M_1$ and $M_3\geqslant M_2$.
\end{itemize}
Moreover, the following inequality holds: $\sup_{\sf c}(x)\leqslant M(\rho_4(x))\leqslant 3\sup_{\sf c}(x)$.
\end{lemma}

%\marg{Modified -- at this moment, I do not add an explanation of bold inequality and bold inequality $\mathbf{M_1>M_3}$ is modified as usual font $M_{1}>M_{3}$}
%We observe that Theorem \ref{theorem:main_classical} follows from the latter: conjugating by $\Delta$ if necessary, 
%we can choose the first factor so that two lines of the matrix have distinct maximal degrees. 

\begin{proof}[Proof of Lemma \ref{lem:classicalLemme}]
The proof is by induction on $r$. %Assume $x=x_1$ is a simple braid. Then a simple calculation shows that 
%all conclusions are still correct, except in the case $x_1=\sigma_1$, where the maximal degree in the first line is only one more than in the third.
A direct calculation shows that all conclusions are correct for the case $r=2$. Here we remark that in the case $r=1$ and  $s_{1}=\sigma_{1}$ ($S(s_1)=\{1\}$), the conclusion does not hold since $M_1(\sigma_1)=M_3(\sigma_1)+1$.

Suppose now $r>2$. Write $x=s_1x'$; by induction $x'$ satisfies the conclusions of the lemma. 
We now distinguish 6 cases, according to the possible values of $S(s_1)$. 
In each case, there are several possibilities for %the precise value of 
$s_1$. Each of them leads to conditions on the 
starting set of $s_2$, the first factor of $x'$, because of the left-weightedness condition on the pair $(s_1,s_2)$.
By induction hypothesis this gives relations between the integers $M'_i:=M_i(x')$. 
In each case, using the explicit computation of $\rho_4(s_1)$, we express the integers $M_i=M_i(s_1x')$ in terms of the $M'_i$ and show that they satisfy the expected relations. In each case, the computations to be performed show that 
$M(\rho_4(x'))+1\leqslant M(\rho_4(s_1x'))\leqslant M(\rho_4(x'))+3$; this shows the last claim in the lemma.

We present the cases $S(s_1)=\{2\}$ and $S(s_1)=\{1,3\}$; this will have the advantage to show the failure in the argument when a factor $\sigma_2\sigma_1\sigma_3$ appears. Other cases are proven similarly.

{\bf Case $S(s_1)=\{2\}$.} The simple element $s_1$ is one of the following: $\sigma_2$, $\sigma_2\sigma_1$, $\sigma_2\sigma_3$,
$\sigma_2\sigma_1\sigma_3\sigma_2$ or $\sigma_2\sigma_1\sigma_3$. We treat two examples; again the three others are dealt with similarly.

{\underline{Suppose $s_1=\sigma_2$.}} Then $F(s_1)=\{2\}$ and by left-weightedness $S(s_2)=\{2\}$. By induction, 
we have $M'_2\geqslant M'_1$ and $M'_2>M'_3$. Multiplying $\rho_4(x')$ on the left by $\rho_4(\sigma_2)=\begin{pmatrix}
1 & q & 0\\
0 & -q & 0\\
0 & 1 & 1\\
\end{pmatrix}$, the new degrees $M_i$ in the product satisfy $M_1=M'_2+1$, $M_2=M'_2+1$ and $M_3\leqslant M'_2$ (possibly the terms of highest degrees in the second and third row of $\rho_4(x')$ cancel with each other). 
Therefore we have $M_2=M_1$ and $M_2>M_3$, thus satisfying the expected conditions when $S(s_1)=\{2\}$.  

{\underline{Suppose $s_1=\sigma_2\sigma_1\sigma_3$}}. 
Then $F(s_1)=\{1,3\}$ and by left-weightedness, $S(s_2)=\{1\},\{3\}$ or $\{1,3\}$.
By induction $M'_3\geqslant M'_1,M'_2$ or $M_1>M_2,M_3$ (with possibly $M_1>M_3+1$). 
Computing $\rho_4(\sigma_2\sigma_1\sigma_3)=\begin{pmatrix} 
0 & q & q^2\\
-q & -q & -q^2\\
1 & 1 & 0\\
\end{pmatrix}$ we get in the first case $M_1=M'_3+2$, $M_2=M'_3+2$ and $M_3\leqslant \max(M'_1,M'_2)\leqslant M'_3$; whence $M_2\geqslant M_1$ and $M_2>M_3$. 
In the second case, unless the strongest inequality $M'_1>M'_3+1$ holds, there is no reason why 
a cancellation could not yield $M_3\geqslant M_2$. Therefore the desired conclusion ($M_2>M_3$) possibly does not hold and we see that the argument fails when $\sigma_2\sigma_1\sigma_3$ is a factor of $x$. 

{\bf Case $S(s_1)=\{1,3\}$.} Then $s_1$ is $\sigma_1\sigma_3$, $\sigma_1\sigma_3\sigma_2$, $\sigma_1\sigma_3\sigma_2\sigma_1$, $\sigma_1\sigma_3\sigma_2\sigma_3$ or $\sigma_1\sigma_3\sigma_2\sigma_3\sigma_2$. 

{\underline{Suppose $s_1=\sigma_1\sigma_3$.}} 
We compute $\rho_4(\sigma_1\sigma_3)=\begin{pmatrix}
-q & 0 &0 \\
1 & 1 & q \\
0 & 0 & -q\\
\end{pmatrix}.$
By left-weightedness, we have $S(s_2)\subset \{1,3\}$ and therefore 
by induction the $M'_i$ satisfy 
$M'_1>M'_2,M'_3$ or $M'_3\geqslant M'_1,M'_2$. 
In the first case $x=\sigma_1\sigma_3x'$ satisfies 
$M_1=M'_1+1$, $M_2\leqslant M'_1$ and $M_3=M'_3+1$: 
we have, as expected, $M_1>M_2,M_3$. In the second case, we have $M_1=M'_1+1,M_2=M'_3+1,M_3=M'_3+1$ whence $M_3\geqslant M_1,M_2$. 

{\underline{Suppose $s_1=\sigma_1\sigma_3\sigma_2$.}}
Then we have to check the product of the matrix $\rho_4(\sigma_1\sigma_3\sigma_2)=\begin{pmatrix}
-q & -q^2 & 0 \\
1 & q & q \\
0 & -q & -q\\
\end{pmatrix}$ by $\rho_4(x')$, where the $M'_i$ satisfy by induction $M'_2\geqslant M'_1$ and $M'_2>M'_3$. 
This gives $M_1=M'_2+2$, $M_2=M'_2+1$ and $M_3=M'_2+1$ whence $M_1>M_2,M_3$. 

{\underline{Suppose $s_1=\sigma_1\sigma_3\sigma_2\sigma_1$.}}
Compute $\rho_4(\sigma_1\sigma_3\sigma_2\sigma_1)=\begin{pmatrix}
0 & -q^2 & 0 \\
0 & q & q\\
-q & -q & -q \\
\end{pmatrix}$. 
On the other hand we have by induction one of the following set of conditions on $x'$: 
$M'_1>M'_2$ and $M'_1>M'_3+1$; or $M'_2\geqslant M'_1$ and $M'_2>M'_3$.
In the first case we obtain $M_1=M'_2+2,M_2\leqslant \max(M'_2+1,M'_3+1)$ and $M_3=M'_1+1$ whence 
$M_3>M_1,M_2$.  In the second case we get $M_1=M'_2+2,M_2=M'_2+1$ and $M_3\leqslant M'_2+1$ whence 
$M_1>M_2,M_3$. 

{\underline{Suppose $s_1=\sigma_1\sigma_3\sigma_2\sigma_3$.}}
Compute $\rho_4(\sigma_1\sigma_3\sigma_2\sigma_3)=\begin{pmatrix}
-q & -q^2 & -q^3\\
1 & q & 0 \\
0 & -q & 0\\
\end{pmatrix}.$
By induction hypothesis, as $S(S_2)\subset F(s_1)=\{2,3\}$, we have $M'_2\geqslant M'_1$ and $M'_2>M'_3$ or $M'_3\geqslant M'_1,M'_2$. 
In the first case: 
\begin{itemize}
\item if $M'_2>M'_3+1$ then $M_1=M'_2+2$, $M_2=M'_2+1=M_3$ whence $M_1>M_2,M_3$,
\item if $M'_2=M'_3+1$ then $M_1\leqslant M'_2+2$ and $M_2=M'_2+1=M_3$ whence we get $M_1>M_2,M_3$ if $M_1=M'_2+2$ and $M_3\geqslant M_1,M_2$ if $M_1<M'_2+2$. 
\end{itemize}
In the second case we obtain $M_1=M'_3+3$, $M_2\leqslant M'_2+1$ and $M_3=M'_2+1$ whence $M_1>M_2,M_3$. 

{\underline{Suppose $s_1=\sigma_1\sigma_3\sigma_2\sigma_1\sigma_3$.}}
The reduced Burau matrix of $s_1$ is 
$\begin{pmatrix}
0 & -q^2 & -q^3\\
0 & q & 0 \\
-q & -q & 0\\
\end{pmatrix}$. 
On the other hand $F(s_1)=\{1,3\}$ whence by induction $x'$ satisfies: 
$M'_1>M'_2,M'_3$ or $M'_3\geqslant M'_1,M'_2$. 
In the first case we get $M_1\leqslant \max(M'_2+2,M'_3+3)$, $M_2=M'_2+1$, $M_3=M'_1+1$. This implies $M_3\geqslant M_1,M_2$ provided 
$M'_1>M'_3+1$ holds. If on the contrary $M'_1=M'_3+1$ we can say more about $M_1$ (actually there will be no cancellation there) because the inequality $M'_1>M'_2$ then implies $M'_3+1>M'_2$ whence $M_1=M'_3+3$. This finally shows $M_1>M_2,M_3$. 
In the second case we obtain $M_1=M'_3+3$, $M_2=M'_2+1$ and $M_3\leqslant \max(M'_1+1,M'_2+1)$ whence 
$M_1>M_2,M_3$.

\end{proof}

\begin{example} 
We show that the conclusion for $S(s_1)=\{2\}$ in Lemma \ref{lem:classicalLemme} does not necessarily hold if 
$s_1=\sigma_2\sigma_1\sigma_3$. 
Indeed, let
%[[U[1],U[2],U[1],U[3]],[U[1],U[3]],[U[1],U[2]],[U[2]]],
$x=\sigma_2\sigma_1\sigma_3\cdot\sigma_1\sigma_3\sigma_2\sigma_1\cdot\sigma_1\sigma_2\sigma_1\sigma_3\cdot\sigma_1\sigma_3\cdot\sigma_1\sigma_2\cdot \sigma_2$. This braid has infimum 0 and is in normal form as written; the degrees of the entries of its Burau matrix are indicated in the following matrix:
%$$\begin{pmatrix}
%-q^{12}+3 q^{11}-2 q^{10}-4 q^9+11 q^8-14 q^7+11 q^6-5 q^5+q^4 & q^{12}+q^{11}-4 q^{10}+4 q^9-q^8-4 q^7+6 q^6-4 q^5+q^4 & q^{13}+q^{12}-3 q^{11}+4 q^{10}-3 q^9+q^8\\
%q^{12}-2 q^{11}-q^{10}+6 q^9-12 q^8+14 q^7-11 q^6+5 q^5-q^4 & -2 q^{11}+3 q^{10}-3 q^9+4 q^7-6 q^6+4 q^5-q^4 & -2 q^{12}+3 q^{11}-4 q^{10}+3 q^9-q^8\\
%-q^{10}+4 q^9-6 q^8+8 q^7-8 q^6+7 q^5-4 q^4+q^3 & -q^{11}+2 q^{10}-q^9+q^7-2 q^6+3 q^5-3 q^4+q^3& -q^{12}+2 q^{11}-2 q^{10}+2 q^9-2 q^8+q^7\\
%\end{pmatrix}.$$
%\ref{lem:classicalLemme}.
%Actually the following $(3\times 3)$-matrix gives the degrees of the terms of highest degree in each entry of the Burau matrix of $x$: 
$\begin{pmatrix} 5 & 8 & 7\\
6 & 7 & 7\\
5 & 7 & 2\\
\end{pmatrix}$.
\end{example}
Lemma \ref{lem:classicalLemme} leads to the following non-vanishing criterion for the reduced Burau representation of 4-braids. 

\begin{theorem}\label{theorem:main_classical}
If the classical left normal form of a 4-braid $x$ %with even infimum 
does not contain a factor $(\sigma_{2}\sigma_{1}\sigma_{3})$ %and starts with a factor $x_1$ with $S(x_1)\neq \{1,3\}$, 
then $\rho_{4}(x) \neq 1$.   
\end{theorem} 

\begin{proof}
Let $N_{\sf c}(x)=\Delta^{p}s_{1}\cdots s_{r}$.
It is easy to check that if $r\leq 2$, then $\rho_4(x) \neq 1$ so we may assume $r > 2$.

First we observe that 
$$ \rho_4(\Delta)=
\begin{pmatrix}
0 & 0 & -q^{3}\\
0 & -q^{2} & 0\\
-q & 0 & 0\\
\end{pmatrix}$$
hence $M_{1}(\Delta x) = M_{3}(x)+3$, $M_{2}(\Delta x) = M_{2}(x)+2$ and $M_{3}(\Delta x) = M_{1}(x)+1$.

Assume that $S(s_{1})\neq \{1,3\}$.
If $p$ is even, then by conjugating by $\Delta$ if necessary, we may assume that $S(s_{1})=\{1\}, \{2\}$, or $\{1,2\}$. By Lemma \ref{lem:classicalLemme}, $\rho_4(s_1\cdots s_r)$ is not an homothety hence $\rho_4(\Delta^ps_1\cdots s_r) \neq 1$. If $p$ is odd, we may assume similarly that $S(s_{1})=\{2\}, \{3\}$, or $\{2,3\}$ hence by Lemma \ref{lem:classicalLemme}, 
$M_2(s_1\cdots s_r)\geqslant M_1(s_1\cdots s_r)$ or $M_3(s_1\cdots s_r)\geqslant M_1(s_1\cdots s_r)$. 
On the other hand, $\rho_4(x)=1$ implies 
$$M_{3}(s_{1}\cdots s_{r})+2 = M_{2}(s_{1}\cdots s_{r})+1 = M_{1}(s_{1}\cdots s_{r}),$$
which is a contradiction.

Now we consider the case $S(s_{1})=\{1,3\}$. Assume for a contradiction that $\rho_4(x)=1$. This implies in particular $M_{i}(yxy')=M_{i}(yy')$ for any 4-braids $y$ and $y'$. We deduce a contradiction by finding appropriate braids $y$ and $y'$. 

\textbf{Case 1:} $s_{1} = \sigma_{1}\sigma_{3}$.
If $p$ is even, put $y=\Delta\sigma_{2}\sigma_{1}\sigma_{3}\sigma_{2}$: $N_{\sf c}((\Delta\sigma_{2}\sigma_{1}\sigma_{3}\sigma_{2})x) = \Delta^{p+2} s_{2}\cdots s_{r}$. By direct calculation 
and under our hypothesis that $\rho_4(x)=Id$,
$$\rho_4(\Delta^{p+2} s_{2}\cdots s_{r})=\rho_4(\Delta\sigma_{2}\sigma_{1}\sigma_{3}\sigma_{2})=
\begin{pmatrix}
-q^3 & 0 & 0\\
q^3 & q^{4} & q^{4}\\
0 & 0 & -q^3\\
\end{pmatrix}.$$
It follows that $M_2(s_2\ldots s_r)=M_1(s_2\ldots s_r)+1=M_3(s_2\ldots s_r)+1$; contradicting Lemma \ref{lem:classicalLemme}
applied with $S(s_2)\subset \{1,3\}$ (which implies in particular $M_2<M_1$ or $M_2\leqslant M_3$). 

If $p$ is even, choose $y=\sigma_{2}\sigma_{1}\sigma_{3}\sigma_{2}$: $N_{\sf c}((\sigma_{2}\sigma_{1}\sigma_{3}\sigma_{2})x) = \Delta^{p+1} s_{2}\cdots s_{r}$. By direct calculation 
$$\rho_4(\Delta^{p+1} s_{2}\cdots s_{r})=\rho_4(\sigma_{2}\sigma_{1}\sigma_{3}\sigma_{2})=
\begin{pmatrix}
0 & 0 & q^{2}\\
-q & -q^{2} & -q^{2}\\
1 & 0 & 0\\
\end{pmatrix}.$$
Hence, $M_{1}(s_2\cdots s_r)=M_{2}(s_1\cdots s_r)=M_{3}(s_2\cdots s_r)+2$, 
contradicting Lemma \ref{lem:classicalLemme} applied with $S(s_2)\subset\{1,3\}$.
%On the other hand, $S(s_{2}) \subset F(\sigma_{1}\sigma_{3})=\{1,3\}$. This contradicts to Theorem \ref{theorem:classicalLemme}. 

\textbf{Case 2:} $s_{1} = \sigma_{1}\sigma_{3}\sigma_{2}$, or %\sigma_{1}\sigma_{3}\sigma_{2}\sigma_{1},
%\sigma_{1}\sigma_{3}\sigma_{2}\sigma_{3},
$\sigma_{1}\sigma_{3}\sigma_{2}\sigma_{3}\sigma_{1}$.

Suppose first that $2 \in F(s_{r})$; then put $y'=\sigma_2$ if $p$ is even and $y'=\Delta\sigma_2$ if $p$ is odd.
Then $$N_{\sf c}(xy')=
\begin{cases}
\Delta^{p}s_{1}\cdots s_{r} (\sigma_{2}) & {\text{if}}\ \ p\ \ \text{is even}\\
\Delta^{p+1}(\Delta^{-1}s_1\Delta)\cdots (\Delta^{-1}s_r\Delta)\sigma_2 & {\text{if}}\ \ p\ \ \text{is odd}\\
\end{cases}$$
Observe that $S(\Delta^{-1}s_1\Delta)=\{1,3\}$. 
Now because of Lemma \ref{lem:classicalLemme} we must have in either case $M_2(xy')<M_1(xy')$ or 
$M_2(xy')\leqslant M_3(xy')$. But on the other hand the calculations 
of $\rho_4(\sigma_2)$ already given as well as $$\rho_4(\Delta\sigma_2)=\begin{pmatrix} 0 & -q^3 & -q^3 \\
0 & q^3 & 0 \\
-q & -q^2 & 0\\
\end{pmatrix}$$ yield the expected contradiction. 

Assume then that $2 \not \in F(s_{r})$. Conjugating by $\Delta$ if necessary, we may assume $1 \in F(s_{r})$ (and $s_1$ is unchanged). 
Then $N_{\sf c}(\sigma_{1}^{-1}x \sigma_{1})=\Delta^ps'_1\ldots s_r\sigma_1$, where $s'_1=(\Delta^{-p}\sigma_1\Delta^p)^{-1}s_1$
satisfies $F(s'_1)=F(s_1)$. 
But as we have already seen $\sigma_{1}^{-1}x \sigma_{1}$ cannot be sent by $\rho_4$ to the identity matrix because $S(s'_1)$ is not $\{1,3\}$.  

\textbf{Case 3:} $s_{1}=\sigma_{1}\sigma_{3}\sigma_{2}\sigma_{1}$, or $\sigma_{1}\sigma_{3}\sigma_{2}\sigma_{3}$

If $ 2 \in F(s_{r})$ then the same argument as Case 2 applies. 
Conjugating by $\Delta$ if necessary, we may assume $1 \in F(s_{r})$. 
\begin{itemize}
\item If $p$ is odd and $s_{1}=\sigma_{1}\sigma_{3}\sigma_{2}\sigma_{1}$, then consider $\sigma_{1}^{-1}x\sigma_{1}$.
\item If $p$ is even and $s_{1}=\sigma_{1}\sigma_{3}\sigma_{2}\sigma_{1}$, then consider $\sigma_{2}x\sigma_{1}\sigma_{2}$.
\item If $p$ is odd and $s_{1}=\sigma_{1}\sigma_{3}\sigma_{2}\sigma_{3}$, then consider $\sigma_{3}^{-1}x\sigma_{1}$.
\item If $p$ is even and $s_{1}=\sigma_{1}\sigma_{3}\sigma_{2}\sigma_{3}$, then consider $\sigma_{1}^{-1}x\sigma_{1}$.
\end{itemize}
In any case, from Lemma \ref{lem:classicalLemme} we obtain a contradiction.
\end{proof}

%\marg{M: When plural should one write "the $M_i$" or "the $M_i$s"?? I: I am not sure, but $M_i$'s may be correct (but maybe it is confusing so one solution is to avoid to use plural...)}

%**************************************************************************************
%From now on and except in the last section of the paper we will deal only with the dual Garside structure, so we will forget about the index ``${\sf dual}$" throughout.

\section{Curve diagrams, the wall-crossing labeling and dual Garside length}\label{section:diagrams}

In this section we review a connection between curve diagrams of braids and the dual Garside structure, which was developed in \cite{iw}. Here we will prove a slightly stronger result which explains how to read the dual normal form of a braid $x$ from its curve diagram.

\subsection{Curve diagrams}

Let $E$ (resp. $\overline{E}$) be the oriented arc in $\mathbb D_{n}$ consisting of the real line segment between $p_1$ and $p_{n}$ (resp. $p_0$ and $p_{n}$). Both line segments $\overline{E}$ and $E$ are oriented from left to right. For $i=0,\ldots, n-1$, we denote by $E_i$ the line segment of $\overline E$ connecting $p_{i}$ and $p_{i+1}$. See Figure \ref{fig:curvediagram} (a); as a convention, the initial segment $E_0$ is depicted as dashed line.

For $i=1,\ldots,n$, let $W_{i}$ be the vertical line segment in $\mathbb D_{n}$, oriented upwards, which connects the puncture $p_{i}$ and the boundary of $\mathbb D_{n}$ in the upper half-disk $\{z \in \mathbb D^{2} \: | \: \textrm{Im}\, z>0 \: \}$. The lines $W_i$ are called the {\em walls}, and their union $\bigcup_{i} W_{i}$ is denoted $W$. Let $U_i$ be a disk-neighborhood of the puncture $p_i$ and set $U = \bigcup_{i}U_{i}$.  See Figure \ref{fig:curvediagram} (b), (c).

\begin{figure}[htbp]
\centerline{\includegraphics[scale=1.2]{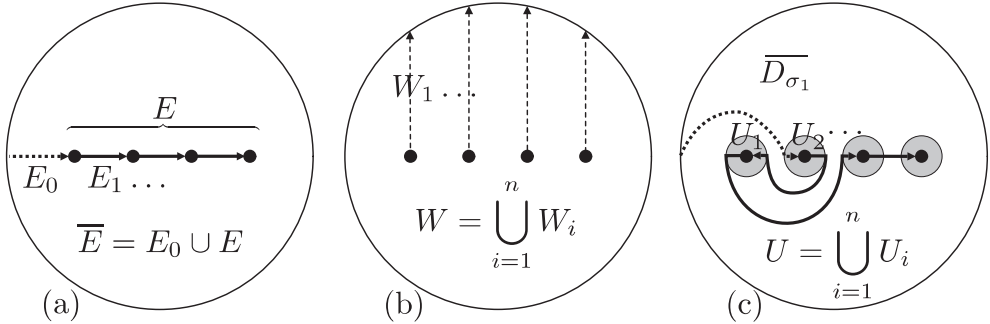}}
   \caption{Curve diagram and walls}
 \label{fig:curvediagram}
\end{figure}

%\marg{Modified since we are using right actions}
The (\emph{total}) \emph{curve diagram} of a braid $x$ is the respective image of $E$ (or $\widebar E$) under a diffeomorphism $\phi$ representing $x$ which satisfies:
\begin{enumerate}
\item $(\widebar E)\phi$ coincides with the real line on $U$,
\item $(\widebar{E})\phi$ is transverse to $W$ and the number of intersections of $(\widebar{E})\phi$ with $W$ is as small as possible (which is equivalent to saying that $(\widebar E)\phi$ and $W$ do not bound together any bigon \cite{fgrrw}).
\end{enumerate}

The (total) curve diagram is uniquely defined up to isotopy of $\mathbb D_n$ that fixes $\partial \mathbb D_n$. We denote by $D_x$ ($\widebar{D_x}$ respectively) the (total) curve diagram of a braid $x$.
Figure \ref{fig:curvediagram} (c) shows the (total) curve diagram of the braid $\sigma_1\in B_4$; according 
to our previous convention, dashed line represents the image of the initial segment $E_0$. 

An \emph{arc segment} (or simply an arc) of the (total) curve diagram $D_x$ (or $\widebar {D_x}$) is a connected component of $D_x-(W\cup U)$ (or $\widebar{D_x}-(W\cup U)$). 
Notice that an arc segment of $\widebar{D_x}$ is in one of the three following cases:
\begin{itemize}
\item it connects two walls $W_i$ and $W_j$,
\item it connects a wall $W_i$ and a puncture $p_j$ (more precisely the neighborhood $U_j$), 
\item it connects two punctures $p_i$ and $p_j$ (more precisely the neighborhoods $U_i$ and $U_j$).
\end{itemize}
In all cases, $i\neq j$ by construction of the curve diagram. We denote such an arc segment, in either case, by $\wideparen{(ij)}$. 
Unless explicitly specified, we will not care about the orientation of an arc segment; this is reflected in our notation. 

\subsection{Wall-crossing labeling and dual normal form}

We now describe the wall-crossing labeling. To that purpose, we need to introduce a modified version of the curve diagrams. 

Let $x\in B_n$. Around each puncture $p_{i}$ distinct from the image of $p_n$ under $x$, we modify the total curve diagram $\widebar {D_{x}}$ inside the neighborhood $U_i$ as shown in Figure \ref{fig:modcurve} (a). 
We denote the resulting (total) curve diagram by $MD_{x}$ $(\widebar{MD_{x}})$, and call it the (\emph{total}) \emph{modified curve diagram} of $x$.
Figure \ref{fig:modcurve} (b) shows the (total) modified curve diagram of $\sigma_{1}\in B_4$.

\begin{figure}[htbp]
 \begin{center}
\includegraphics*[scale=1]{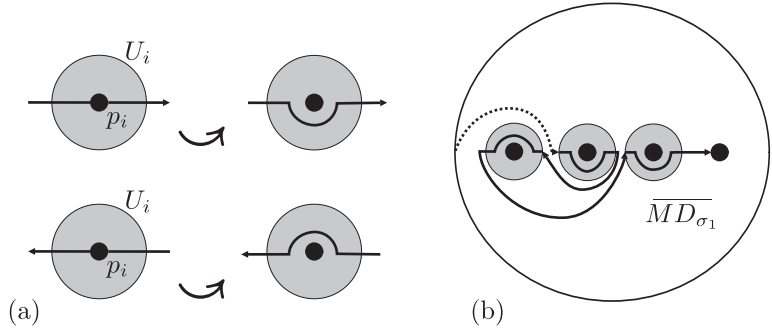}
   \caption{Modified curve diagrams}
 \label{fig:modcurve}
  \end{center}
\end{figure}

Take a smooth parametrization of $\widebar{MD_x}$, viewed as the image of a function $\gamma\co [0,1] \rightarrow \mathbb D^{2}$. For each connected component $\alpha$ of $\widebar{MD_{x}}- (W\cup  U)$, we assign the algebraic intersection number of $W$ and the arc $\gamma([0,v])$, where $v \in [0,1]$ is taken so that $\gamma(v) \in \alpha$. 
Notice that a connected component of $\widebar{MD_{x}}-(W\cup  U)$ naturally corresponds to an arc segment of $\widebar {D_x}$, since 
$\widebar{MD_{x}}$ and $\widebar{D_{x}}$ are identical except on $U$. This allows to attribute a label to each arc segment of $\widebar {D_x}$; this integer-valued labeling is called the \emph{wall-crossing labeling} of $x$. We define
$\LWcr(x)$ and $\SWcr(x)$ as the largest and smallest possible labels occuring in the wall-crossing labeling for arc segments \emph{in the curve diagram} $D_x$, respectively.  
% and the label of an arc segment $\alpha$ is denoted by $\Wcr(\alpha)$. 
%\marg{M:removed the notation $\Wcr(\alpha)$ because this depends on $x$..otherwise $\Wcr_{x}(\alpha)$ maybe..
%Back to the oldest way for defining $\LWcr$ and $\SWcr$
%T: Thanks to the condition of curve diagram, $\Wcr(\alpha)$ is well-defined in the sense that for any other choice of representative, say $x'$,  we have a corresponding arc segment $\alpha'$ with the same Wcr. Anyway, it is OK to avoid to use $\Wcr(\alpha)$.  
%}
%Figure \ref{fig:labelex} shows an example of wall-crossing labeling.
%\begin{figure}[htb]
% \begin{center}
%\includegraphics*[width=100mm]{labelex_0_1.eps}
%\caption{The total curve diagram of the braid $\beta=(\sigma_1\sigma_2^{-1})^2$, and its wall crossing labeling. Among the labels of the solid (not dashed) arcs, the smallest one is $-2$ and the largest one is $2$.}
%\label{fig:labelex}
%  \end{center}
%  \end{figure}
%\marginpar{\tiny the braid in Figure 3 was wrong}
%
%\subsection{Wall-crossing labeling and dual normal form}
%
%$$\LWcr(x)=\max\{r\in \mathbb Z,\ \: D_x \  {\textrm{has an arc segment labeled }}  r\},$$
%$$\SWcr(x)=\min\{r\in \mathbb Z,\ \: D_x \  {\textrm{has an arc segment labeled }}  r\}.$$

Notice that to define $\LWcr$ and $\SWcr$, we used the largest and smallest labels only of the curve diagram $D_{x}$, not the total curve diagram $\widebar{D_{x}}$. However, in order to determine the wall crossing labelings we need to consider the total curve diagram. 

The following relates the wall-crossing labeling with the dual length of a braid:
\begin{theorem}\cite[Theorem 3.3]{iw}
\label{theorem:wall}
For a braid $x \in B_{n}$, we have the following equalities: 
\begin{enumerate}
\item $\sup_{\sf d}(x) = \LWcr(x)$.
\item $\inf_{\sf d}(x) = \SWcr(x)$.
%\item[(iii)] $\ell_{\sf d}(x) =  \max\{0,\LWcr(x)\} - \min\{0,\SWcr(x)\}$.
\end{enumerate}
\end{theorem}

Here we show a stronger result than Theorem \ref{theorem:wall}, which is suggested by and is implicit in the proof of \cite[Theorem 3.3]{iw}: one can read not only supremum, infimum, but also dual Garside normal form from the curve diagram.
Recall from Section \ref{section:garsidedual} the lattice ordering $\preccurlyeq_{\sf d}$ on $B_n$. 

\begin{theorem}
\label{theorem:finalfactor}
Let $x \in B_{n} $ be a braid and put $\ell=\LWcr(x)-\SWcr(x)$.
For $k=\ell,\ldots,1$, we define $d_{k}$ inductively as follows: 
\begin{enumerate}
\item $d_{\ell}$ is the least common multiple (with respect to $\preccurlyeq_{\sf d}$) of all letters $a_{i,j}$ such that the curve diagram $D_{x}$ contains an arc segment $\wideparen{(i j)}$ with wall-crossing labeling $\LWcr(x)$.
\item $d_{k}$ is the least common multiple (with respect to $\preccurlyeq_{\sf d}$) of all letters $a_{i,j}$ such that the curve diagram $D_{x d_{\ell}^{-1}\cdots d_{k+1}^{-1}}$ contains an arc segment $\wideparen{(i j)}$ with wall-crossing labeling $(k+\SWcr(x))$.
\end{enumerate}
Then the dual normal form of $x$ is given by
$$ N_{\sf d}(x)=\delta^{\SWcr(x)}d_1\ldots d_{\ell}.$$.
\end{theorem}

Before proving Theorem \ref{theorem:finalfactor}, we review from \cite{iw} the description of how the action of a dual simple element  affects the curve diagram of a braid and its wall-crossing labeling. This was the key of the proof of Theorem \ref{theorem:wall}.

Dealing with the dual Garside structure, it will be convenient to work with the  
model of the punctured disk described in Section \ref{section:garsidedual}; in that context the wall $W_i$ is the shortest straight segment connecting the puncture $p_i$ to the boundary, oriented outwards. Notice that the isotopy involved in the change of model for the punctured disk does not affect the wall-crossing labeling since the latter 
is defined in terms of algebraic intersection of arcs and walls.

Let $x\in B_n$; let $d$ be a dual simple element. Write $d=P_1\cdots P_r$ the decomposition of $d$ into a product of disjoint polygons. For $i=1,\ldots,r$, let $N_i$ be a regular neighborhood of the polygon $P_i$ in $\mathbb D_n$. Let $A_i$ be an annulus which is a regular neighborhood of the boundary of $N_i$. Suppose moreover that $A_i$ is chosen so that none of its two boundary components forms a bigon together with the walls $W$ or the diagram $D_{x}$ and so that as many intersection points of $D_{x}$ and $W$ as possible lie in $A_i$.

Now $D_{xd}$ and its wall-crossing labeling are obtained as follows. The respective actions of each of the polygons $P_i$ are independent; each of them acts non-trivially only on the inner complementary component of the corresponding annulus and on the annulus itself (where the diagram just describes a spiral). 
For each $i=1,\ldots,r$, $N_i$ is turned by one notch in the clockwise direction and all labels are increased by one; on the annulus $A_i$, $D_{xd}$ and the corresponding labels are interpolated linearly; see Figure \ref{fig:Paction}. 
The action of the inverse of a dual simple element can be described in a very similar way, the twisting on $N_i$ being in the opposite direction, and all labels being decreased by one. 

\begin{figure}[hbtp] 
 \begin{center}
\includegraphics*[scale=0.5, width=110mm]{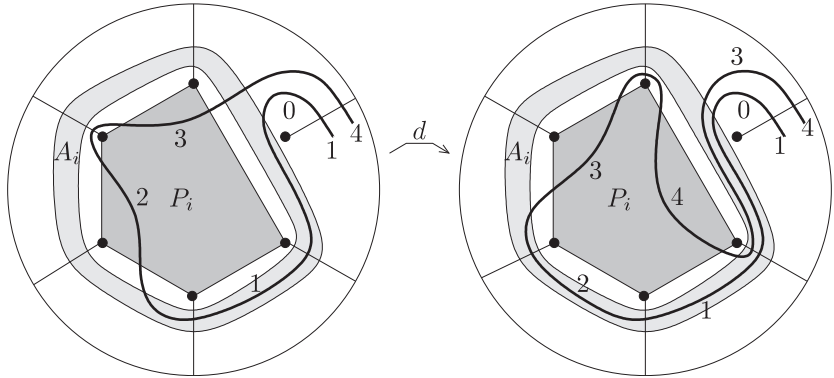}
\caption{How to draw curve diagram of $xd$ from $D_{x}$}\label{fig:Paction}
  \end{center}
\end{figure}

\begin{proof}[Proof of Theorem \ref{theorem:finalfactor}]
We prove the theorem by induction on $\ell=\LWcr(x)-\SWcr(x)$. 
When $\ell \leqslant 1$, the result is explicitly contained in the proof of Lemma 3.5 in \cite{iw}. 
Suppose that $\ell\geqslant 2$. By induction, it is sufficient to show that $(d_{\ell-1},d_{\ell})$ is left-weighted. 

We check the left-weightedness using Proposition \ref{prop:left-weighted}.
Write the dual simple elements $d_{\ell}$ and $d_{\ell-1}$ as products of disjoint polygons: $d_{\ell-1}=P_1\cdots P_{r_{\ell-1}}$ and $d_{\ell}=Q_{1}\cdots Q_{r_{\ell}} $, respectively.
Let $i,j$ be two vertices of some polygon $Q \in \{Q_1,\ldots,Q_{r_{\ell}}\}$. We must show that there exists a polygon $P \in \{ P_1,\ldots, P_{r_{\ell}}\}$ having vertices $k,l$ such that $a_{k,l}\vdash a_{i,j}$.  
By definition of $d_{\ell-1}$, it is sufficient to show that $D_{xd_{\ell}^{-1}}$ admits an arc segment $\wideparen{(kl)}$ with label $\LWcr(x)-1$ and such that $a_{k,l}\vdash a_{i,j}$.

Assume first that the diagram $D_x$ admits an arc segment $\wideparen{(ij)}$ with label $\LWcr(x)$. Then according to the description above of the action of the inverse of a polygon, the diagram $D_{xd_{\ell}^{-1}}$ admits an arc segment $\wideparen{(kl)}$ with label $\LWcr(x)-1$ such that $k\in (j,i-1)$ and $l\in(i,j-1)$, as desired. Moreover, we notice that if $g_{i,j}$ and $h_{i,j}$ are the rightmost vertex of $Q$ in $(j,i-1)$ and $(i,j-1)$ respectively, then 
$k\in (g_{i,j},i-1)$ and $l\in (h_{i,j},j-1)$. See Figure \ref{Fig:LW} (a).

Assume now that $D_x$ does not have an arc segment $\wideparen{(ij)}$ with label $\LWcr(x)$.
Since both $i$ and $j$ are vertices of $Q$, by definition of $Q$ there must exist arc segments $\wideparen{(ib)},\wideparen{(jc)}$ of $D_{x}$ with label $\LWcr(x)$, for some punctures $b, c \not \in \{i,j\}$, possibly $b=c$.
 
Suppose that such a puncture $b$ can be chosen so that 
$a_{b,i}a_{i,j}$ is a dual simple braid. This means that $b\in (i+1,j-1)$. 
But we have just seen that the action of $Q^{-1}$ produces an arc segment $\wideparen{(kl)}$ labeled by $\LWcr(x)-1$ in the diagram $D_{xd_{\ell}^{-1}}$, such that $k\in (j,i-1)$ (because the rightmost vertex of $Q$ in the arc $(b+1,i-1)$ certainly lies in the subarc $(j,i-1)$) and $l\in (i,b-1)\subset (i,j-1)$). Similarly, if $c$ can be chosen so that $a_{c,j}a_{i,j}$ is a dual simple braid, we get a pair of punctures $k,l$ with the expected property. See Figure \ref{Fig:LW} (b).

\begin{figure}[hbtp] 
 \begin{center}
\includegraphics*[scale=0.5, width=120mm]{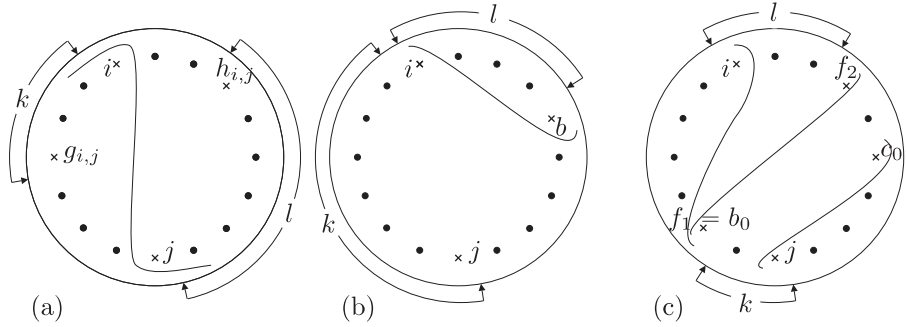}
\caption{Proof of Theorem \ref{theorem:finalfactor}; all arc segments represented are labeled $\LWcr(x)$, crosses indicate vertices of the polygon $Q$.}\label{Fig:LW}
  \end{center}
\end{figure}

Finally, suppose that no arc segment $\wideparen{(bi)}$ nor $\wideparen{(jc)}$ with labeling $\LWcr(x)$ of $D_x$ has the above property. 
Then $b\in (j+1,i-1)$, $c\in (i+1,j-1)$ and $b\neq c$. Among all $b$ so that $D_x$ admits an arc $\wideparen{(bi)}$ labeled $\LWcr(x)$, let $b_0$ be the leftmost one. Similarly, among all $c$ so that $D_x$ admits an arc $\wideparen{(jc)}$ labeled $\LWcr(x)$, let $c_0$ be the leftmost one. 
The punctures $i,j,b_0$ and $c_0$ are all distinct and vertices of the polygon $Q$. By definition of $d_\ell$, there must exist 
an arc segment $\wideparen{(f_1f_2)}$ in $D_{x}$ with labeling $\LWcr(x)$ such that $f_1\in (j+1,b_0)$, $f_2\in (i+1,c_0)$; the punctures $f_1,f_2$ are also vertices of $Q$. 
But then $D_{xd_{\ell}^{-1}}$ admits an arc $\wideparen{(kl)}$ labeled $\LWcr(x)-1$ with $k\in(j,f_1-1)$ and $l\in(i,f_2-1)$, thus with the required property. See Figure \ref{Fig:LW} (c) (an example where $f_1=b_0$).
This completes the proof of Theorem \ref{theorem:finalfactor}.
\end{proof}

%We notice that we could have written a more formal proof of Theorem \ref{theorem:finalfactor}, using ``words associated to curve diagrams" in the same spirit as in \cite{calvez}, where words are attached to simple closed curves in $\mathbb D_n$. However we believe this would make the argument unnecessarily much longer and more technical than the present one.  

%*******************************************************************************
\section{Burau representation}\label{section:burau}

In this section we review a homological construction of the Burau representation; 
this interpretation is used to relate the latter with the wall-crossing labeling.

\subsection{The Burau representation}
Fix the base point $\ast=-(n+1) \in \partial \mathbb D_{n}$ on the boundary of $\mathbb D_n$. The fundamental group $\pi_{1}(\mathbb D_{n})$ is a free group of rank $n$ where the free generator $x_{i}$ is represented by a loop which rounds the $i$th puncture $p_i$ once clockwise.
Let $\epsilon: \pi_{1}(\mathbb D_{n}) \rightarrow \Z=\langle q \rangle$ be the homomorphism which sends all $x_{i}$ to the generator $q$. Geometrically, for a loop $\gamma$, $\epsilon([\gamma])$ is the sum of the algebraic winding number of $\gamma$ about the puncture points $\{p_{i}\}$ (in the clockwise direction).

Let $\pi: \widetilde{\mathbb D_{n}} \rightarrow \mathbb D_{n}$ be the infinite cyclic covering corresponding to $\textrm{Ker}(\epsilon)$, and fix a lift $\widetilde{\ast}$ of the base point.
The group of covering transformations of $\widetilde{\mathbb D_{n}}$ is identified with the cyclic group $\left<q\right>$.
Then $H_{1}(\widetilde{\mathbb D_{n}};\Z)$ can be endowed with a structure of $\Z[q,q^{-1}]$-module, where multiplication by $q$ corresponds to the deck transformation. Moreover it turns out that $H_{1}(\widetilde{\mathbb D_{n}};\Z)$ is free of rank $(n-1)$ as a $\Z[q,q^{-1}]$-module. Since $\epsilon$ is $B_{n}$-invariant, we have a linear representation 
\[ \rho: B_{n} \rightarrow \GL (H_{1}(\widetilde{\mathbb D_{n}};\Z) ). \]
This is called the {\em (reduced) Burau representation}. In the rest of this section, we keep the same notation $\epsilon$, $\widetilde{\mathbb D_n}$ and $\widetilde{\ast}$ for the above defined winding number evaluation morphism, covering space of $\mathbb D_n$ and base point.

\subsection{Forks}\label{section:forks}
Let $Y$ be the $Y$-shaped graph  
consisting of three external vertices: a distinguished one $r$, two others $v_{1}$ and $v_{2}$ and one internal vertex~$c$ and three edges relating 
each external vertex to the internal one (see Figure \ref{fig:fork} (a)). 
We orient the edges of $Y$ as shown in Figure \ref{fig:fork} (a).

A {\em fork} is an embedded image of $Y$ into $\mathbb D_n$ such that:
\begin{itemize}
\item All points of $Y- \{r,v_1,v_2\}$ are mapped to the interior of $\mathbb D_n$.
\item The distinguished vertex $r$ is mapped to the base point $\ast$.
\item The other two external vertices $v_{1}$ and $v_{2}$ are mapped to two different puncture points.
\end{itemize}

Given a fork $F$, the image of the edge $[r,c]$ is called the {\em handle} of~$F$ %and denoted by $H(F)$; 
and the image of $[v_1,v_2]=[v_{1},c] \cup [c,v_{2}]$, regarded as a single oriented arc, is called the {\em tine} of $F$ and denoted by $T(F)$. The image of $c$ is called the \emph{branch point} of~$F$.
Figure \ref{fig:fork} (b) shows a fork $F$ (with the handle depicted in grey line and the tine in black line).

\begin{figure}[htbp]
 \begin{center}
\includegraphics*[scale=0.5, width=100mm]{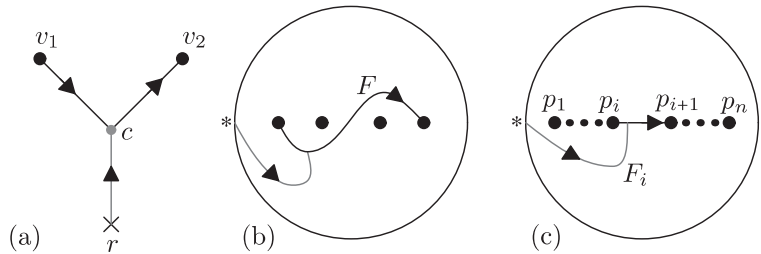}
   \caption{Fork and standard fork $F_i$}
 \label{fig:fork}
  \end{center}
\end{figure}

For a fork $F$, let $\gamma \co [0,1] \rightarrow \mathbb D_{n}$ be the handle of $F$, viewed as a path in $\mathbb D_{n}$ and take a lift 
\[ \widetilde{\gamma} \co [0,1] \rightarrow \widetilde{\mathbb D_{n}} \]
of $\gamma$ so that  $\widetilde{\gamma}(0)=\widetilde{\ast}$.
Let $\Sigma(F)$ be the connected component of $\pi^{-1}(T(F))$ that contains the point $\widetilde{\gamma}(1)$. 
The \emph{homology class of $H_{1}(\widetilde{\mathbb D_n};\Z)$ represented by} $F$ is then defined as the homology class represented by $\Sigma(F)$. 
By abuse of notation, we still denote this homology class by $F$.
Strictly speaking, since $\Sigma(F)$ is not compact we need to work with the homology of locally finite chains $H_{1}^{lf}(\widetilde{\mathbb D_n};\Z)$ or $H_{1}(\widetilde{\mathbb D_n},\widetilde{P};\Z)$, where $\widetilde{P}$ is the preimage of a small neighborhood of the punctures in $\mathbb D_{n}$. Rigorous treatments are well-known and give rise to the same conclusions (see \cite{bi3}, for example), so we do not take care of these subtle points.

Of special importance is the following family of particularly simple forks: 
for $i=1,\ldots,n-1$, let $F_{i}$ be the fork whose tine is a straight arc connecting the $i$th and the $(i+1)$st punctures and whose handle is contained in the lower half of the disk $\mathbb D_{n}$ (see Figure \ref{fig:fork} (c)). 
These are called {\em standard forks}.
Standard forks $F_{1}, \ldots, F_{n-1}$ form a basis of $H_{1}(\widetilde{\mathbb D_{n}};\Z)$. 
The group $\GL(H_{1}(\widetilde{\mathbb D_{n}};\Z))$ can be identified with $\GL(n-1;\Z[q,q^{-1}])$ using the basis of standard forks. 
This allows to get the familiar matrix description of the reduced Burau representation:

\[
\rho_n(\sigma_{1}) = \left(
\begin{array}{cc} 
-q & 0 \\
1 & 1 \\
\end{array}
\right) \oplus I_{n-3}, \;\;\; 
\rho_n(\sigma_{n-1})=I_{n-3}\oplus
\left(\begin{array}{cc}
1 & q\\
0 &-q\\
\end{array}
\right),
\] 
\[
\rho_n(\sigma_{i}) = I_{i-2} \oplus 
\left(
\begin{array}{ccc} 
1 &  q & 0 \\
0 & -q & 0 \\
0 & 1  & 1 \\
\end{array}
\right) 
\oplus I_{n-i-2}, \;\;\; (i=2,\ldots,n-2)
\]

%
%\marg{
%M: Why? Using $\langle N,F\rangle$, Lemma 5.1 would yield: $\rho_n(x)_{i,j}=\langle N_j,F_i^{z}\rangle$. I: I just thought it is natural to adopt a notation so that $\rho_n(x)_{ij}=\langle (F_i)z,N_j,\rangle$ (it is less confusing, I think.)}

\subsection{The noodle-fork pairing}
A {\em noodle} is an embedded oriented arc in $\mathbb D_n$ which begins at the base point $\ast$ and ends at some point of the boundary $\partial\mathbb D_n$. 
Noodles represent relative homology classes in $H_1(\widetilde{\mathbb D_n}, \partial  \widetilde{\mathbb D_n};\Z)$. 

The \emph{noodle-fork pairing} (in our notation, it should say fork-noodle pairing) is a homology intersection (algebraic intersection) pairing 
$$\langle\;,\;\rangle:H_1(\widetilde{\mathbb D_n} ;\Z) \times H_1(\widetilde{\mathbb D_n}, \partial  \widetilde{\mathbb D_n};\Z) \rightarrow \Z[q,q^{-1}].$$
Geometrically, it is computed in the following way (see \cite{bi3} Section 4).

Given a fork $F$ and a noodle $N$, put $T(F)$ and $N$ transverse with minimal intersections.
Let $z_1,\ldots,z_r$ be the intersection points. Each intersection point $z_{i}$ then contributes a monomial $\varepsilon_{i}q^{e_i}$ to $\langle N,F\rangle$, where $\varepsilon_i$ is the sign of the intersection between $T(F)$ and $N$ at $z_{i}$ and $e_i$ is an integer.
The noodle-fork pairing is then given by
\[ \langle F,N\rangle = \sum_{1 \leq i \leq r} \varepsilon_{i}q^{e_i} \in \Z[q,q^{-1}]. \]

The integer $e_i$ is computed as follows. Let $\gamma_i$ be the loop which is the composition of three paths $A$, $B$ and $C$ in $\mathbb D_{n}$:
\begin{itemize}
\item $A$ is a path from $\ast$ to the branch point of~$F$ along the handle of~$F$.
\item $B$ is a path from the branch point of $F$ to $z_{i}$ along the tine $T(F)$.
\item $C$ is a path from $z_{i}$ to $\ast$ along the noodle $N$.
\end{itemize}
Then $e_{i} = \epsilon([\gamma_i])$: that is, $d_i$ is the sum of the winding numbers of the loop $\gamma_i$ about the puncture points $p_1,\ldots,p_n$.

As for forks, we define a distinguished family of noodles: for $i=1,\ldots,n-1$, the \emph{standard noodle} $N_i$ is the noodle which has empty intersection with the walls and ends at some boundary point between $W_{i}$ and $W_{i+1}$. 
Given a braid $x$, the entries of its Burau matrix can be computed using the noodle-fork pairing in a fairly direct manner.

\begin{lemma}[Burau Matrix formula]
 \label{lemma:Buraumatrix}
Let $x\in B_n$. 
Then for $1\leqslant i,j\leqslant n$, the entry $\rho_n(x)_{ij}$ of its Burau matrix is given by $\rho_n(x)_{ij}=\langle (F_{i})x,N_{j} \rangle$.
\end{lemma}
\begin{proof}
By definition, $(F_i)x = \sum_{k=1}^{n-1}  F_k \rho_n(x)_{ik} \in H_{1}(\widetilde{\mathbb D_n};\Z)$, hence for $i,j \in 1,\ldots,n-1$ we have
\[ \langle (F_i)x , N_j \rangle = \sum_{k=1}^{n-1}  \langle F_{k},N_{j} \rangle \rho_n(x)_{ik}\]
It is directly checked that $\langle F_{k},N_{j}\rangle = \delta_{kj}$ (Kronecker's delta) hence
\[ \langle (F_{i})x,N_{j} \rangle = \rho_n(x)_{ij}. \]
\end{proof}

\begin{example}
As an example of application of Lemma \ref{lemma:Buraumatrix}, we can retrieve the Burau matrices associated to Artin generators $\sigma_i$. First, we notice that for $k=1,\ldots,n-1$, $i\neq k-1,k,k+1$, $(F_i)\sigma_k=F_i$, so that $\langle (F_i)\sigma_k,N_j\rangle=\delta_{i,j}$.
For the remaining values of $i$, Figure \ref{fig:example} shows the images $(F_i)\sigma_k$. 
\begin{figure}[htbp]
 \begin{center}
\includegraphics*[scale=0.5, width=100mm]{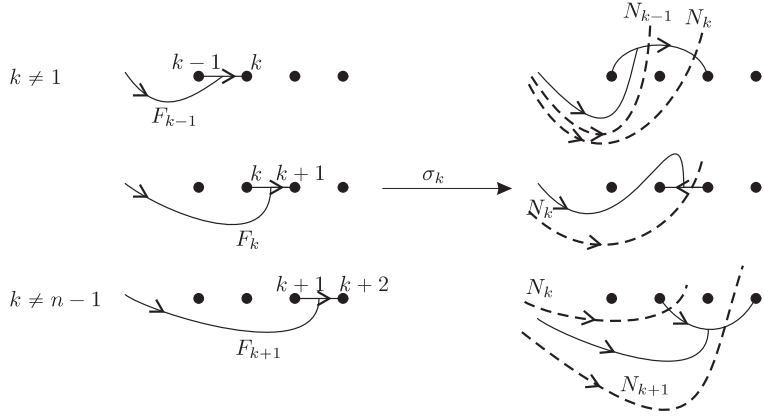}
   \caption{On the left part, forks $F_{k-1}$, $F_k$ and $F_{k+1}$ and on the right part, their images under the action of the braid $\sigma_k$; relevant noodles are depicted in dashed lines.}
 \label{fig:example}
  \end{center}
\end{figure}

With the help of Figure \ref{fig:example} we can conclude:

\begin{tabular}{lr}
{$\langle (F_{k-1})\sigma_k,N_j\rangle=\begin{cases} 
0 & {\text{if}}\  j<k-1\ {\text{or}}\  j>i,\\
1 & {\text{if}}\  j=k-1,\\
q & {\text{if}}\  j=k.\\
\end{cases}$} & 

{$\langle (F_{k+1})\sigma_k,N_j\rangle=\begin{cases} 
0 & {\text{if}}\ j<k\ {\text{or}}\  j>k+1,\\
1 & {\text{if}}\  j=k,\\
1 & {\text{if}}\  j=k+1.\\
\end{cases}$} \\
\end{tabular}

$$\langle (F_{k})\sigma_k,N_j\rangle=
\begin{cases} 
0 & {\text{if}}\  j\neq k,\\
-q & {\text{if}}\  j=k .\\
\end{cases}
$$
Lemma \ref{lemma:Buraumatrix} then allows to retrieve the matrices given at the end of Section \ref{section:forks}.

\end{example}

\subsection{Noodle-fork pairing and wall-crossing labeling}
We finally review a connection between the integers $e_i$ in the computation of the noodle-fork pairing and the wall-crossing labeling. 
This will yield the expected relation between the Burau representation and the wall-crossing labeling. 

Let $x\in B_n$. First we recall how to assign wall-crossing labelings for points belonging to the image $(F_i)x$ of the standard fork $F_i$ under $x$. 
Let us consider the part of the curve diagram $D_{x}$ that is the image of  $E_{i}$ (the line segment between the $i$-th and $(i+1)$-st punctures). We identify this part $(E_{i})x$ of the curve diagram with $(T(F_{i}))x$. Moreover, a part of the modified curve diagram  can naturally be regarded as the handle of $(F_{i})x$, as shown in Figure~\ref{fig:cdtofork}. This identification induces the wall crossing labeling on each connected component of $(F_{i})x- (W\cup U)$. For a point $z \in (F_i)x- (W\cup U)$ we denote by $\Wcr_x(z)$ the corresponding label. 

%\marg{Please replace $F_{3,4}$ with $F_{3}$ in the Figure.}
\begin{figure}[htbp]
 \begin{center}
\includegraphics*[scale=0.5, width=80mm]{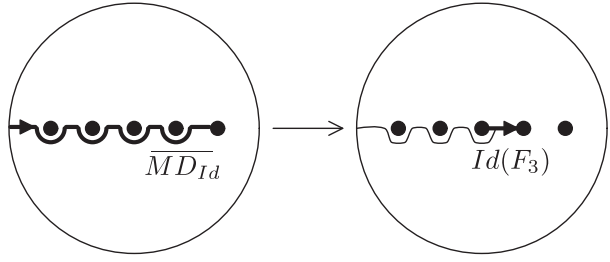}
\caption{Viewing a curve diagram as a union of tines of forks, and viewing initial segments of modified curve diagrams as handles.}
\label{fig:cdtofork}
\end{center}
\end{figure}

Let $N$ be a noodle; we may assume that no intersection point in $(T(F_i))x \cap N$ belongs to $W\cup U$. 

\begin{lemma}
\label{lemma:exponent}
Fix an intersection point $z\in (T(F_i))x \cap N$.
Let $c(z)$ be the algebraic intersection number of $W$ and the path $C$ in the definition of the pairing $\langle (F_i)x,N\rangle$
(i.e. $C$ is a path from $z$ to $\ast$ along $N$). Let $e(z)$ be the degree of $q$ in the $z$-contribution to $\langle (F_i)x,N\rangle$.
Then
\[ e(z) = \Wcr_x(z) + c(z)\]
\end{lemma}
\begin{proof}
Let $A$ and $B$ be the paths in the definition of the pairing $\langle (F_i)x,N\rangle$. Then $\Wcr(z)$ is nothing but the algebraic intersection number of $W$ and the composite path $BA$. Hence the algebraic intersection number of $W$ and the loop $\gamma = CBA$ is $\Wcr_x(z) + c(z)$, which is, by definition, equal to $e(z)= \epsilon(\gamma)$.
\end{proof}

\begin{corollary}
\label{cor:ineqn}
For any braid $x\in B_n$, the following inequality holds:
$$M(\rho_n(x))\leqslant \sup \!{}_{\sf d}(x).$$
\end{corollary}
\begin{proof}
For we have, by definition and thanks to Lemma \ref{lemma:Buraumatrix},
$$M(\rho_n(x)) = \max_{i,j}\{ M(\langle (F_{i})x,N_{j} \rangle) \}.$$
For a standard noodle $N_j$ and a point $z\in (T(F_{i}))x \cap N_j$ the integer $c(z)$ in Lemma \ref{lemma:exponent} is always $0$ because standard noodles do not intersect with walls.
Therefore we have 
$$\max_{i,j}\{ M (\langle (F_{i})x,N_{j} \rangle) \}\leq \LWcr(x)$$
and finally, as $\LWcr(x) = \sup_{\sf d} (x)$ (Theorem \ref{theorem:wall}) we are done.
\end{proof}

%******************************************************************************

\section{Braids whose Burau Matrix detects the dual Garside normal forms}\label{section:dual}

In view of Corollary \ref{cor:ineqn}, a natural question is to ask when the converse inequality holds. Theorem \ref{theorem:main_dual}
will give a sufficient condition for the maximal degree appearing in the Burau matrix of a braid to be equal to its dual supremum. Actually we will prove more: under the same condition, it is possible to determine the dual normal form from the Burau matrix.

To state Theorem \ref{theorem:main_dual} we first introduce the notion of simply-nestedness as a refinement of the left-weightedness condition (Proposition \ref{prop:left-weighted}), which will allow us to get a better control on the action of a braid in dual normal form. 

%\begin{prop-defn}
%Let $d,d'$ be two dual simple elements, expressed as products of disjoint polygons $d=P_1\ldots P_r$ and $d'=Q_1\ldots Q_{r'}$ respectively. 
%Then the two following are equivalent:
%\begin{itemize}
%\item For any two vertices $i,j$ of some polygon among $Q_{1},\ldots, Q_{r'}$, there exists a {\em unique} polygon among $P_1,\ldots, P_{r}$ having two vertices $k,l$ such that $a_{k,l}\vdash a_{i,j}$.
%%$k\in (j,i-1)$ and $l\in (i,j-1)$. 
%\item For any polygon $Q$ among $Q_{1},\ldots, Q_{r'}$, there exists a unique polygon among $P_1,\ldots, P_{r}$ such that 
%for any pair of vertices $i,j$ of $Q$, there exists some pair of vertices $k,l$ of $P$ with $a_{k,l}\vdash a_{i,j}$. 
%%$k\in (j,i-1)$ and $l\in (i,j-1)$.
%\end{itemize}
%If one of these equivalent conditions is satisfied we say that the pair $(d',d')$ is simply-nested. 
%\end{prop-defn}
%
%\begin{proof} 
%The second condition is clearly stronger than the first so we show that $1\Rightarrow 2$. 
%Let $Q$ be a polygon among $Q_{1},\ldots, Q_{r'}$; let $i,j$, $i',j'$ be two pair of vertices of $Q$; by (i) there exist
%a unique polygon $P$ in $d$ having two vertices $k,l$ such $a_{k,l}\vdash a_{i,j}$ and there exists a unique polygon $P'$ in $d$ having two vertices $k',l'$ such that $a_{k',l'}\vdash a{i',j'}$. We must show that $P=P'$.   
%\end{proof}
Let $d,d'$ be two dual simple elements, expressed as products of disjoint polygons $d=P_1\ldots P_r$ and $d'=Q_1\ldots Q_{r'}$ respectively.  
We say that the ordered pair $(d,d')$ is \emph{simply-nested} if  for any polygon $Q$ among $Q_{1},\ldots, Q_{r'}$, there exists a {\em unique} polygon $P$ among $P_1,\ldots, P_{r}$ such that for any two vertices $i,j$ of $Q$, the polygon $P$ has two vertices $k,l$ such that 
$a_{k,l}\vdash a_{i,j}$. 
%$k\in (j,i-1)$ and $l\in (i,j-1)$. 
A braid $x$ will be said to be \emph{simply-nested} if each pair of consecutive factors in its dual normal form is simply-nested.
%For $n=4$, the pair $(a_{1,2}a_{3,4},a_{2,4})$ is not simply-nested since on one hand $1\in (4,1)$ and $2\in(2,3)$ and on the other $3\in(2,3)$ and $4\in (4,1)$. Neither is simply-nested the pair $(a_{1,2}a_{3,4}, a_{3,4}a_{2,3})$. 

Let $B_{n}^{\sf sn}$ be the set of simply nested $n$-braids.
Although $B_{n}^{\sf sn}$ does not form a group, $B_{n}^{\sf sn}$ is a regular language over the alphabet $[1,\delta]$.
We also remark that $B_{n}^{\sf sn}$ is not symmetric: $x \in B_{n}^{\sf sn}$ does not imply $x^{-1} \in B_{n}^{\sf sn}$. A simple example is the 4-braid $x=(a_{3,4})(a_{2,4})$. Although $x$ is simply-nested, $N_{\sf d}(x^{-1}) = \delta^{-2}(a_{1,2}a_{3,4})(a_{1,2}a_{1,4})$ which is not simply nested.

%
%Geometrically speaking, Proposition \ref{prop:left-weighted} says that the pair $(d,d')$ is left-weighted if and only if each polygon $Q_j$ ``intersects'' with some polygon $P_i$. In general, for a fixed $Q_{j}$ there may be several different polygons $P_{i}$ intersecting it.
%To control the action of braids in dual left normal form, we consider the following situation which is a refinement of the notion of left-weightedness. 
%
%\begin{definition}
%Let $d=P_{1}\cdots P_{r_1}$ and $d'=Q_1\cdots Q_{r_2}$ be two dual simple elements expressed as products of disjoint polygons.
%The pair $(d,d')$ is said to be \emph{simply-nested} if for any two vertices $i,j$ of a polygon among $Q_{1},\ldots, Q_{r_2}$, there exists a {\em unique} polygon among $P_1,\ldots, P_{r_1}$ having two vertices $k,l$ such that $k\in (j,i-1)$ and $l\in (i,j-1)$. 
%We will say that a braid $x$ is \emph{simply-nested} if each pair of consecutive factors in its dual normal form is simply-nested.
%\end{definition}
%
%For example, if $d$ is a polygon, then the pair $(d,d')$ is simply-nested if and only if it is left-weighted.
%
%Let $B_{n}^{\sf sn}$ be the set of simply nested $n$-braids.
%We remark that although $B_{n}^{\sf sn}$ does not form a group, $B_{n}^{\sf sn}$ is a regular language over the alphabet $[1,\delta]$.

We now can state our second main result:

\begin{theorem}
\label{theorem:main_dual}
Let $x \in B_n$ be a simply-nested braid.
\begin{enumerate}
\item[(i)] $\sup_{\sf d}(x) = M(\rho_n(x))$.
\item[(ii)] One can compute the dual normal form from the matrix $\rho_{n}(x)$, so the restriction of the Burau representation on the set of simply-nested braid $B_{n}^{\sf sn}$ is injective.
\end{enumerate} 
\end{theorem}

For a braid $x$ and $i=1,\ldots,n-1$ let $\mathcal M_{x}(E_i)$ be the set of the arc segments of $(E_i)x$ whose wall-crossing labeling attains the maximal value $\LWcr(x)$ (possibly empty). We say that two arc segments in the curve diagram are \emph{parallel} if both are described by $\wideparen{(ij)}$ for some $i,j$. 
We consider the following property \textbf{(C)} (\emph{Coherence property}) for a braid $x$:
\begin{definition}
Let $x\in B_n$ and $N_{\sf d}(x) = \delta^{p}d_1\cdots d_r$. Express $d_{r}$ as a product of disjoint polygons: $d_{r}=Q_{1}\cdots Q_{b}$.
We say that \emph{$x$ has the property} \textbf{(C)} if for each $i=1,\ldots,n-1$, any two arc segments $\alpha$ and $\alpha'$ in $\mathcal {M}_{x}(E_i)$ intersecting a common polygon $Q \in \{Q_{1},\ldots,Q_{b}\}$ are parallel and have the same orientation.
\end{definition}

\begin{lemma}\label{lemma:CondC}
If $x$ has the property \textbf{(C)}, then $\sup_{\sf d}(x)= M(\rho_n(x))$ holds.
\end{lemma}

\begin{proof}%[Proof of Claim]
Let $N_{\sf d}(x) = \delta^{p}d_1\cdots d_r$. 
Take $i$ so that $\mathcal{M}_{x}(E_i)$ is non-empty. Take the minimal number $k$ so that there exists an arc segment $\alpha = \wideparen{(kp)} \in \mathcal{M}_{x}(E_i)$ for some $p>k$. We look at the entry $\rho_n(x)_{ik}$ in the Burau matrix of $x$, which is equal to $\langle (F_i)x, N_{k}\rangle$ by Lemma \ref{lemma:Buraumatrix}. In view of Corollary \ref{cor:ineqn} and Theorem \ref{theorem:wall}, the desired equality will be shown provided $M_q(\langle (F_i)x,N_{k}\rangle)=\LWcr(x)$. 

Let $\alpha'$ be another arc segment in $\mathcal{M}_{x}(E_i)$ which intersects the noodle $N_k$. By minimality of $k$, $\alpha'=\wideparen{(ku)}$ for some $u\in (k+1,n)$. By Theorem \ref{theorem:finalfactor}, %$x_r$ contains a polygon 
some polygon $Q$ in the decomposition of $d_r$  %\in \{Q_{1},\ldots,Q_{t}\}$ 
has vertices $k,p,u$; both arcs $\alpha$ and $\alpha'$ intersect $Q$. Hence by property \textbf{(C)}, $\alpha$ and $\alpha'$ are parallel with the same orientation (notice that, in particular, $u=p$ holds).  
This shows that all arcs in $\mathcal M_{x}(E_i)$ intersecting the noodle $N_k$ have the same sign of intersection so $M_q (\langle (F_i)x,N_{k}\rangle)=\LWcr(x)$. 
\end{proof}

\begin{lemma}\label{lemma:SNimplies_C}
Let $x\in B_n^{\sf sn}$. Then $x$ has Property {\rm{\textbf{(C)}}}. 
\end{lemma}
\begin{proof}
The proof is by induction on the number $r$ of non-$\delta$ factors in the dual normal form of $x$. 
The case $r=1$ is checked by direct calculation. Actually, in this case $(E_j)x$ has at most one maximal labeled arc for any $j$. 

Suppose $N_{\sf d}(x)=\delta^{p}d_1\cdots d_r$ with $r>1$. Then $x'=\delta^{p} d_1\cdots d_{r-1}$ is also simply-nested and has the Property \textbf{(C)} by induction hypothesis. Let us express
%$Q_1,\ldots, Q{l}$ be the disjoint polygons involved in the decomposition of $x_r$.
$d_{r-1}$ and $d_{r}$ as products of disjoint polygons: $d_{r-1}=P_{1}\cdots P_{b'}$ and $d_{r}=Q_{1}\cdots Q_{b}$.

For $f=1,\ldots,n-1$, suppose that $\alpha=\wideparen{(ij)}$ and $\alpha'=\wideparen{(i'j')}$ are two arcs in $\mathcal{M}_{x}(E_f)$ that intersect a common polygon $Q \in \{Q_{1},\ldots,Q_{b}\}$.  
By Theorem \ref{theorem:finalfactor}, all of $i,i',j,j'$ are vertices of $Q$.
Following the proof of Theorem \ref{theorem:finalfactor} we can find arcs $\beta=\wideparen{(kl)},\beta'=\wideparen{(k'l')}$ in the diagram $D_{x'}$ with label $\Wcr(x)-1$ ($\beta,\beta'\in \mathcal M_{x'}(E_f)$) and $a_{k,l}\vdash a_{i,j}$ and $a_{k',l'}\vdash a_{i',j'}$. Moreover we can choose $\beta,\beta'$ so that $\alpha$ and $\alpha'$ come from $\beta$ and $\beta'$ respectively under the action of $Q$ (see Figure \ref{Fig:LW} (a)). 

By simply-nestedness assumption, $k,l,k',l'$ must be vertices of a common polygon $P \in \{P_{1},\ldots,P_{b'}\}$. This implies that both $\beta$ and $\beta'$ intersect with the same polygon $P$, hence by Property \textbf{(C)} for $x'$, the arc segments $\beta$ and $\beta'$ are parallel with the same direction. Therefore the same property holds true for $\alpha$ and $\alpha'$, as we wanted to show.
%under the action of the polygon $Q$, arc segments $\alpha$ and $\alpha'$ appear as the images of arc segments in $\mathcal{M}_{x'}(E_{f})$ that intersect $Q$.
%We will show that all arc segments in $\mathcal{M}_{x'}(E_{f})$ intersecting $Q$ are parallel with the same orientation. This proves that $\alpha$ and $\alpha'$ are also parallel with the same orientation, as needed.
\end{proof} 
 
\begin{remark}
We observe that, although it is a stronger property, simply-nestedness is fairly easy to check whereas checking Property {\bf{(C)}} directly is often a hard task since we need to know both dual normal form and the curve diagram of braids.
\end{remark}

\begin{proof}[Proof of Theorem \ref{theorem:main_dual}]
Lemmas \ref{lemma:CondC} and \ref{lemma:SNimplies_C} show part (i).

We explain how to compute the final factor $d_{r}$ of the dual normal form of $x$, which gives an algorithm to compute the whole dual normal form of $x$ from its Burau matrix. Let $N_{\sf d}(x)=\delta^{p}d_{1}\cdots d_{r}$ and write $d_{r}$ as a product of disjoint polygons: $d_{r}=Q_{1}\cdots Q_{b}$.

Our strategy to determine $d_{r}$ is as follows.
We show how to find some $a_{i,j}$ satisfying $a_{i,j} \preccurlyeq_{\sf d} d_{r}$ from $\rho_{n}(x)$. Since $d_{r}$ is written as a product of at most $(n-2)$ letters $a_{i,j}$, by iterating this procedure at most $(n-2)$ times, we eventually determine $d_{r}$.

For $i=1,\ldots,n-1$, let $M^{c}_{i}(x)=\max\{\rho_n(x)_{ji}\: | \: j=1,\ldots,n-1\}$, namely, the maximal degrees of the variable $q$ in the $i$-th \emph{column} of the Burau matrix of $x$ (do not confuse $M_{i}(x)$ in Section \ref{section:classical}, where we used the maximal degrees of the $i$-th \emph{row}. First we show that $M^{c}_{i}(x)$ gives candidates of $a_{i,j}$ satisfying $a_{i,j}\preccurlyeq_{\sf d} d_{r}$.

\begin{claim}
\label{claim:1p}
We have $$\min\{i\in \{1,\ldots,n\}\: | \: \exists j,a_{i,j}\preccurlyeq_{\sf d}d_r\}=
\min\{i\in \{1,\ldots,n-1\}\: | \: M^{c}_i(x)=M(\rho_n(x))\}.$$ 
\end{claim}
\begin{proof}
Let $i_0=\min\{i\in \{1,\ldots,n\}\: | \: \exists j,a_{i,j}\preccurlyeq_{\sf d}d_r\}$.
Let $k>i_0$ be such that $a_{i_0,k}\preccurlyeq d_r$. 

First, we show that $M_{i_0}^{c}(x)=M(\rho_n(x))$. 
Since $a_{i_0,k}\preccurlyeq d_r$ and by Theorem \ref{theorem:finalfactor} there must exist some $p\in\{1,\ldots,n\}$, $p>i_0$, such that $D_x$ admits an arc $\alpha=\wideparen{(i_0p)}$ labeled $\LWcr(x)=M(\rho_n(x))$. Let also $Q\in\{Q_1,\ldots,Q_b\}$ having vertices $i_0,p,k$ and let $j$ be such that $\alpha\in\mathcal M_x(E_j)$.
We observe that $\alpha$ intersects the noodle $N_{i_0}$. We will show that $M(\rho_n(x)_{j,i_{0}})=M(\rho_n(x))$.
Indeed, let $\alpha'\in \mathcal M_x(E_j)$ and suppose that $\alpha'$ intersects $N_{i_0}$. By minimality of $i_0$, $\alpha'$ must intersect with the polygon $Q$ and by Property {\bf{(C)}}, $\alpha'$ is parallel to $\alpha$ with the same orientation. Hence the $\alpha$ and $\alpha'$ intersect with $N_{i_0}$ with the same sign.  
Therefore $M(\langle (F_j)x,N_{i_0}\rangle)=M(\rho_n(x))$ as we wanted to show. 

Second, we show that for $i<i_0$, $M^{c}_i(x)\neq M(\rho_n(x))$. 
Otherwise, there would exist some $j$ such that $\mathcal M_x(E_j)$ is non-empty and we could find some $\beta\in\mathcal M_x(E_j)$ intersecting the noodle $N_{i}$. But then because of Theorem \ref{theorem:finalfactor}, $\beta$ yields a letter prefix of $d_r$ which contradicts the minimality of $i_0$. 
\end{proof}

It follows that we can find $i_0$ as above looking at the columns of $\rho_n(x)$. 
We then proceed to find $k$ such that $a_{i_0,k}\preccurlyeq_{\sf d}d_r$.
Let $j$ be such that $M(\rho_n(x)_{j,i_0})=M_{i_0}^c(x)$. There might be several ones, we just choose any of them. 
Then there is a maximally labeled arc segment $\alpha=\wideparen{(i_0p)}\in \mathcal M_x(E_j)$ which intersects the noodle $N_{i_0}$. 
It is enough to determine $p$ because Theorem \ref{theorem:finalfactor} implies that $a_{i_0,p}\preccurlyeq_{\sf d}d_r$. 
Notice that, by Property {\bf{(C)}}, $p$ is unique with the property that $(E_j)x$ contains a maximally labeled arc segment of the form $\wideparen{(i_0p)}$. In the remaining part of the proof, $p$ and $\alpha=\wideparen{(i_0p)}$ are fixed and we explain how to determine $p$ from the Burau matrix.

\begin{claim}
\label{claim:p,p-1}
The integer $p$ above satisfies 
%\begin{gather}

%\label{eqn:k}
\begin{itemize}
\item[(i)] $M(\langle(E_{j})x, N_{p-1}\rangle)=M(\rho_{n}(x)_{j (p-1)})=\LWcr(x),$
\item[(ii)] $M(\langle (E_{j})x, N_{p}\rangle)=M(\rho_{n}(x)_{j p})\neq \LWcr(x).$
\end{itemize}
%\end{gather}
\end{claim}

\begin{proof}
(i) First let $\alpha'$ be any arc segment in $\mathcal M_x(E_j)$ intersecting the noodle $N_{p-1}$.
By minimality of $i_0$ it must also intersect the chord segment joining punctures $i_0$ and $p$, hence the polygon $P$. 
By Property {\bf{(C)}}, $\alpha'$ is parallel to $\alpha$ with the same direction. This shows (i). 

(ii) Consider now an arc $\alpha'\in (E_j)x$ which intersects the noodle $N_p$. We show that its label is strictly less than $\LWcr(x)$. Otherwise, by minimality of $i_0$, $\alpha'$ would also intersect the polygon $P$; by Property {\bf{(C)}} it would be parallel to $\alpha$, contradicting the fact that it intersects $N_p$. 
\end{proof}
 
Now we notice that simply looking at the matrix $\rho_n(x)$ is not sufficient to find $p$: there might be several integers sharing with $p$ the properties of Claim \ref{claim:p,p-1}. However let $\{p_1,\ldots,p_c\}$ be the set of those punctures satisfying conditions of Claim
\ref{claim:p,p-1} and suppose $p_1<\cdots<p_c$.  
To find $p$, we compute matrices $\rho_n(xa_{i_0,p_{\iota}}^{-1})$ for $\iota=1,\ldots,c$, until we find $M(\langle (E_j)xa_{i_0,p_{\iota}}^{-1},N_{p_{\iota}}\rangle))<\LWcr(x)$. This determines $p$ thanks to the following observation:

\begin{claim}
For $\iota=1,\ldots,c$ the integer $p_{\iota}$ satisfies:
$$\begin{cases}
M(\langle (E_j)xa_{i_0,p_{\iota}}^{-1},N_{p_{\iota}}\rangle)=\LWcr(x) & \text{if}\ p_{\iota}<p,\\
M(\langle (E_j)xa_{i_0,p_{\iota}}^{-1},N_{p_{\iota}}\rangle)<\LWcr(x) & \text{if}\ p_{\iota}=p.\\
\end{cases}$$
\end{claim}
\begin{proof}
Let $\iota\in\{1,\ldots,c\}$ be such that $p_{\iota}<p$. 
We observe that each maximally labeled arc segment in $(E_j)xa_{i_0,p_{\iota}}^{-1}$ which intersects 
the noodle $N_{p_{\iota}}$ corresponds to a maximally labeled arc segment in $(E_j)x$ which intersects $N_{p_{\iota}}$ (in the same sign); see Figure \ref{fig:claim67}. This shows: 
$$M(\langle (E_j)xa_{i_0,p_{\iota}}^{-1},N_{p_{\iota}}\rangle)=M(\langle (E_j)x,N_{p_{\iota}}\rangle)=\LWcr(x).$$

\begin{figure}[htbp]
 \begin{center}
\includegraphics*[scale=1,width=120mm]{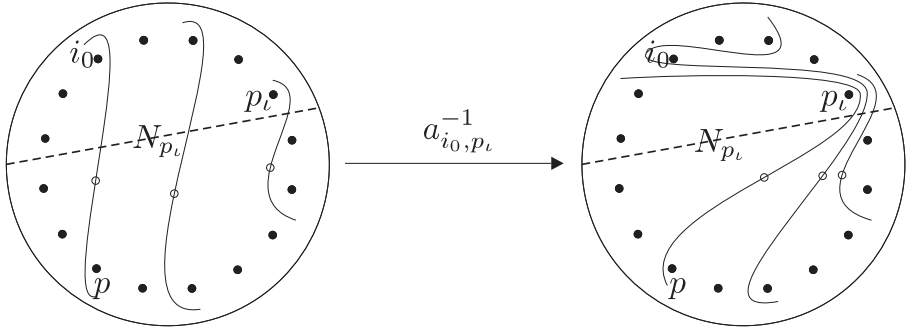}
\caption{Indicated by a small circle, maximally labeled arcs in $(E_j)x$ (in $(E_j)xa_{i_0,p_{\iota}}^{-1}$ respectively) having intersection with the noodle $N_{p_{\iota}}$ (depicted as dashed line).}
\label{fig:claim67}
\end{center}
\end{figure}
On the other hand, no arc segment with maximal label in $(E_j)xa_{i_0,p}^{-1}$ intersects the noodle $N_p$, so we get 
the desired equality $M(\langle (E_j)xa_{i_0,p}^{-1},N_{p}\rangle)<\LWcr(x)$.
\end{proof}

This achieves the proof of part (ii) of Theorem \ref{theorem:main_dual}.
\end{proof}

Before proving the corollaries of Theorem \ref{theorem:main_dual}, we make some remarks on the proof. 

\begin{remark}
We notice that only Property {\bf{(C)}} is needed in the proof of the first part of Theorem \ref{theorem:main_dual}, as well 
as in the procedure aiming to determine the last factor of the dual normal form from the matrix. However in order to use this procedure in an inductive way and hence determine the whole of the dual normal form, the simply-nestedness assumption is crucial because Property \textbf{(C)} just concerns the last factor $d_{r}$ so it does not guarantee that $\delta^{p}d_{1}\cdots d_{r-1}$ also has Property \textbf{(C)}. Moreover as we saw, simply-nestedness is often much easier to recognize as Property {\bf{(C)}}.
\end{remark}

\begin{remark}\label{remark:PropC'}
A statement similar to the first part of Theorem \ref{theorem:main_dual} concerning the dual infimum and the minimal degree of the entries of the Burau matrix, although it sounds quite reasonable, cannot be deduced from our proof. Indeed, the simply-nestedness assumption as well as the Property \textbf{(C)} do not control intersections of noodles and arc segments will smallest wall-crossing labeling at all. For the same argument to work, we need the following analogue of Property \textbf{(C)}:
\begin{description}
\item[(C$'$)] 
Let $N_{\sf d}(x) = \delta^{p}d_1\cdots d_r$. Express $d_{r}$ as a product of disjoint polygons: $d_{r}=Q_{1}\cdots Q_{b}$.
For each $i=1,\ldots,n-1$, any two arc segments $\alpha$ and $\alpha'$ in $x(E_i)$ labeled $\SWcr(x)$ and intersecting a common polygon $Q \in \{Q_{1},\ldots,Q_{t}\}$ are parallel and have the same orientation.
\end{description}

This makes a good contrast with the case of the LKB representations \cite{iw}; in that context one can apply key techniques of treating arc segments with the largest crossing labeling (Bigelow's key Lemma \cite[Lemma 5.1]{bigelowlinear}) to arc segments with the smallest crossing labeling as well.
\end{remark}

We now proceed to show the corollaries to Theorem \ref{theorem:main_dual}.

For the 3-braid group, we have the following.

\begin{corollary}\label{cor:3strands}
Let $x\in B_3$. Then 
\begin{enumerate}
\item[(i)] $\sup_{\sf d}(x) = M(\rho_{3}(x))$,
\item[(ii)] $\inf_{\sf d}(x) = m(\rho_{3}(x))$.
\item[(iii)] One can compute the dual normal form of $x$ from the matrix $\rho_{3}(x)$.
\end{enumerate}
\end{corollary}
\begin{proof}
We just need to be careful about the assertion (ii) (see Remark \ref{remark:PropC'}). To prove (ii), we notice that all dual simple elements are represented as a connected polygon. This observation and an argument similar to Lemma \ref{lemma:SNimplies_C} shows any 3-braid has property \textbf{(C$'$)}.
\end{proof}
%\marg{it looks it comes rather from the description of arcs with smallest label given in \cite{iw}
%I: What I and Bert showed in \cite{iw} is that SWcr is preseved -- I am not sure an argument in \cite{iw} gurantee the property \textbf{(C$'$)}.}

\begin{corollary}\label{cor:4strands}
Let $x\in B_4$ and $N_{\sf d}(x)=\delta^pd_1\cdots d_r$. Assume that for all $i=1,\ldots,r-1$,
$(d_i,d_{i+1})$ is not in the following list:
$$ \left\{ \begin{array}{c}(a_{1,2}a_{3,4},a_{2,4}),(a_{1,2}a_{3,4},a_{3,4}a_{2,3}),(a_{1,2}a_{3,4},a_{1,2}a_{1,4}),\\
\ (a_{2,3}a_{1,4},a_{1,3}), (a_{2,3}a_{1,4},a_{1,3}a_{2,3}),(a_{2,3}a_{1,4},a_{1,3}a_{1,4}) \end{array}\right\}$$
%
%\begin{tabular}{cc}
%$\{(a_{1,2}a_{3,4},a_{2,4}),(a_{1,2}a_{3,4},a_{3,4}a_{2,3}),(a_{1,2}a_{3,4},a_{1,2}a_{1,4}),$ & \\
% & $(a_{2,3}a_{1,4},a_{1,3}), (a_{2,3}a_{1,4},a_{1,3}a_{2,3}),(a_{2,3}a_{1,4},a_{1,3}a_{1,4})\}.$
% \end{tabular}
%$\{(a_{1,2}a_{3,4},a_{2,4}),(a_{1,2}a_{3,4},a_{3,4}a_{2,3}),(a_{1,2}a_{3,4},a_{1,2}a_{1,4}),$ 
%\begin{flushright}{$(a_{2,3}a_{1,4},a_{1,3}), (a_{2,3}a_{1,4},a_{1,3}a_{2,3}),(a_{2,3}a_{1,4},a_{1,3}a_{1,4})\}.$}
 % \end{flushright}
Then 
\begin{enumerate}
%\item $\inf_{\sf d}(x) = m_{q}(\rho_{4}(x))$,
\item[(i)] $\sup_{\sf d}(x) = M_{q}(\rho_{4}(x))$,
\item[(ii)] one can compute the dual normal form of $x$ from the matrix $\rho_4(x)$. 
\end{enumerate}
In particular, if the dual left normal form of a 4-braid $x$ 
does not contain a factor $(a_{1,2}a_{3,4})$ or $(a_{2,3}a_{1,4})$ then $\rho_{4}(\beta) \neq 1$.
\end{corollary}
\begin{proof}
It follows from Theorem \ref{theorem:main_dual} observing that the mentionned pairs are the only six ones which are left-weighted and not simply-nested.
\end{proof}

Let $e:B_{n} \rightarrow \Z$ be the abelianization map, given by $e(\sigma_{i})=1$.
Since $\det(\rho_n(x))=q^{e(x)}$, if $\rho_n(x)=Id$ then $e(x)=0$.
By combining this simple constraints, we get a useful criteria for braids not to lie in the kernel of Burau representation.

\begin{corollary}\label{cor:nstrands}
Let $x\in B_{n}$ be a non-trivial braid and $N_{\sf d}(x)=\delta^{p}d_{1}\cdots d_{r}$. If there exists $r'\leqslant r$ such that
\begin{enumerate}
\item[(i)] The subword $x_{r'}=\delta^{p}d_{1}\cdots d_{r'}$ is simply-nested,
\item[(ii)] $r' > e(d_{r'+1}\cdots d_{r})$,
\end{enumerate}
then $\rho_n(x)\neq 1$.
Moreover the condition (ii) is always satisfied if $r' > \frac{n-2}{n-1}r$.
\end{corollary}
\begin{proof}
Put $E=e(d_{r'+1}\cdots d_{r})$. Assume contrary, $\rho_n(x)=1$.
Since $e(d_{i})\leqslant(n-2)$, we have 
$ 0 = e(x) \leqslant (n-1)p + (n-2)r' + E $ so $-p \leqslant \frac{1}{n-1}((n-2)r'+ E)$.
On the other hand, by (i)
\[ 0=M(\rho_n(x)) = M(\rho_{n}(\delta^{p}d_{1}\cdots d_{r'}) \rho_n(d_{r'+1} \cdots d_{r})) \geqslant M(\rho_{n}(\delta^{p}d_{1}\cdots d_{r'}) )= p+r' \]
hence $r' \leqslant -p$. Therefore $r' \leqslant  \frac{1}{n-1}((n-2)r'+ E)$, which is equivalent to $r' \leqslant E$. This contradicts to (ii). The last assertion follows from the inequality $E \leqslant (n-2)(r-r')$. 
\end{proof}

We close this section by looking at some known examples of elements in the kernel of the Burau representations $\rho_5$ and $\rho_6$. 

Consider the braids $$x=[v_2^{-1}v_1\sigma_3v_1^{-1}v_2,\sigma_3] \in B_6,$$ where $v_1=\sigma_1\sigma_2^{-1}\sigma_5^{-1}\sigma_4$ and 
$v_2=\sigma_1^{-2}\sigma_2\sigma_5^{2}\sigma_4^{-1}$ and 
$$y=[w_1^{-1}\sigma_4w_1,w_2^{-1}\sigma_4\sigma_3\sigma_2\sigma_1\sigma_1\sigma_2\sigma_3\sigma_4w_2]\in B_5,$$ where 
$w_1=\sigma_3^{-1}\sigma_2\sigma_1^{2}\sigma_2\sigma_4^{3}\sigma_3\sigma_2$ and $w_2=\sigma_4^{-1}\sigma_3\sigma_2\sigma_1^{-2}\sigma_2\sigma_1^{2}\sigma_2^{2}\sigma_1\sigma_4^{5}$. 

It is known that $\rho_5(y)=Id$ and $\rho_6(x)=Id$. 
The following are dual normal forms of a conjugate $x'$ and $y'$ of $x$ and $y$, respectively: 

$$N_{\sf d}(x')=\delta_6^{-6}(a_{1,6}a_{4,5})(a_{1,6}a_{2,5})(a_{1,6}a_{4,6}a_{2,3})(a_{1,5}a_{4,5}a_{2,3})(a_{3,6}a_{4,5})(a_{1,6}a_{2,5}a_{4,5})(a_{1,6}a_{3,5})(a_{1,6}a_{5,6}a_{2,4})$$
$$(a_{1,3}a_{5,6})(a_{2,4}a_{5,6})(a_{1,3}a_{5,6}a_{4,5})(a_{2,6}a_{4,5})(a_{1,3}),$$

$$N_{\sf d}(y') =\delta_5^{-23} (a_{2,5}a_{4,5})(a_{1,5}a_{3,5})(a_{1,4}a_{3,4})(a_{2,5})(a_{1,5}a_{2,3})(a_{1,5}a_{3,4})^{2}(a_{1,3})(a_{2,5}a_{3,4})$$
$$(a_{1,4}a_{3,4})(a_{1,2}a_{1,4})(a_{1,2}a_{3,5})(a_{1,2}a_{1,5}a_{3,4})(a_{1,5})(a_{1,2})(a_{2,3})(a_{3,4})$$
$$(a_{2,4})(a_{1,3}a_{4,5})(a_{1,2}a_{4,5})(a_{2,3}a_{4,5})(a_{1,3}a_{4,5})(a_{1,2}a_{3,5}a_{4,5})(a_{2,5}a_{3,5})$$
$$(a_{1,3}a_{1,4})(a_{1,2}a_{1,4})(a_{1,2}a_{1,3}a_{4,5})(a_{1,2}a_{4,5})(a_{2,3}a_{4,5})^{2}(a_{2,5})(a_{1,4}a_{2,3})$$
$$(a_{2,5}a_{3,5})(a_{1,5}a_{3,5})(a_{1,5}a_{2,4})(a_{1,5}a_{4,5}a_{2,3})(a_{1,5}a_{3,5}a_{2,3})^{4}(a_{2,4})(a_{1,3}a_{4,5})$$
$$(a_{1,2}a_{4,5})(a_{2,3}a_{4,5})(a_{1,3}a_{4,5})(a_{1,2}a_{4,5}a_{3,4}).$$

See Figure \ref{fig:NormalForm6} for pictorial (polygon) expression of $N_{\sf d}(x')$. One notices that $N_{\sf d}(x')$ contains many non-simply-nested pairs. Similarly, one observes that $N_{\sf d}(y')$ also contains a lot of non-simply-nested pairs.
These examples and our results on simply-nested braid suggest the Burau matrix of a braid $x$ is close to be the identity matrix only when its dual normal form contains many non-simply nested pairs.

\begin{figure}[htbp]
 \begin{center}
\includegraphics*[scale=1.4]{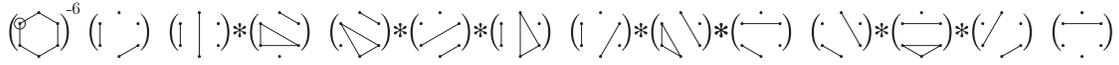}
\caption{A pictorial way to represent the dual normal form of the braid $x'\in B_6$ (the puncture rounded by a circle is the first); the symbol $\ast$ represents non simply-nested pairs.}
\label{fig:NormalForm6}
\end{center}
\end{figure}

%*********************************************************************************************************


\begin{thebibliography}{1} 
\bibitem[Bi]{bigelow} S.\ Bigelow, {\em The Burau representation of the braid group $B_n$ is not faithful for $n=5$,} Geometry and Topology 3 (1999), 397-404.
\bibitem[Bi2]{bigelowlinear} S.\ Bigelow, {\em Braid groups are linear,} J. Amer. Math. Soc. (2) 14 (2001), 471-486.
\bibitem[Bi3]{bi3} S.\ Bigelow, {\em Does the Jones polynomial detect the unknot?,} J. Knot Theory Ramifications (4) 11 (2002), 493-505.

\bibitem[Bir]{birman} J.\ Birman, {\em Braids, Links, and Mapping Class Groups,} Annals of Math. Studies \textbf{82}, Princeton Univ. Press (1974).

\bibitem[BGGM1]{bggm1} J.\ Birman,V.\ Gebhardt and J.\ Gonz\'{a}lez-Meneses, {\em Conjugacy in Garside Groups I: Cycling, Powers and Rigidity,} Groups Geom. Dyn. (3) 1 (2007), 221-279. 
\bibitem[BKL]{bkldual} J.\ Birman, K.-H.\ Ko and S.-J.\ Lee, {\em A new approach to the word and conjugacy problems in the braid groups,} Adv. Math. 139 (2) (1998), 322-353.
%\bibitem[Bu]{burau}  W.\ Burau, {\em \"{U}ber Zopfgruppen und gleischsinnig verdrillte Verkettungen,} Abh. Math. Sem. Ham. II (1936), 171-178.
%\bibitem[C]{calvez} M. Calvez, {\em Dual Garside structure and reducibility of braids,} J. Algebra (1) 356 (2012), 355-373.
\bibitem[EM]{em} E.\ ElRifai and H.\ Morton, {\em Algorithms for positive braids,} Quart. J. Math. Oxford. Ser. (2) 45 (1994), no. 180, 479-497.
\bibitem[ECHLPT]{echlpt} D.B.A. Epstein, J. Cannon, D. Holt, S. Levy, M. Paterson and W. Thurston, {\em Word processing in groups,} Jones and Bartlett Publishers, Boston, MA, 1992.
\bibitem[FM]{farbmargalit} B. Farb and D. Margalit, {\em A primer on Mapping Class Groups,} Princeton Mathematical Series, 2011. 
\bibitem[FGRRW]{fgrrw} R. Fenn, M. Greene, D. Rolfsen, C. Rourke and B. Wiest, {\em Ordering the braid groups,} Pacific J. Math. 191 (1999), 41-74
\bibitem[Gar]{garside} F.\ Garside, {\em The braid groups and other groups,} Quart. J. Math. Oxford Ser. (2) 20 (1969), 235-254.
\bibitem[IW]{iw} T.\ Ito and B.\ Wiest, {\em Lawrence-Krammer-Bigelow representation and dual Garside length of braids,} arXiv:1201.0957v1
%\bibitem[I]{i} T.\ Ito, {\em Reading the dual Garside length of braids from homological and quantum representations,} arXiv:1205.5245
\bibitem[Kra]{krammer} D.\ Krammer, {\em Braid groups are linear,} Ann. Math. 155 (2002), 131-156.
\bibitem[Law]{lawrence} R.\ Lawrence, {\em Homological representations of the Hecke algebra,} Comm. Math. Phys. 135 (1990), 141-191. 
\bibitem[LS]{leesong} S.-J.\ Lee and W.-T.\ Song, {\em The kernel of Burau(4)$\otimes \mathbb Z_p$ is all pseudo-Anosov,} Pacific J. Math. (2) 219 (2005),  303-310.
\bibitem[LP]{lp} D.\ Long and M.\ Paton, {\em The Burau representation of the braid group $B_n$ is not faithful for $n\geqslant 6$,} Topology 32 (1993), 439-447.
\bibitem[MP]{magnuspeluso} W.\ Magnus and A.\ Peluso, {\em On a theorem of V.I. Arnold,} Com. on Pure and App. Math 23 (1969), 683-692.
\bibitem[Mo]{moody} J.\ Moody, {\em The Burau representation of the braid group $B_n$ is not faithful for large $n$,} Bull. Amer. Math. Soc.25 (1991), 379-384.
%\bibitem[Wie]{wiest} B.\ Wiest, {\em How to read the length of a braid from its curve diagram,} Groups Geom. Dyn. 5 (2011), 673-681.
\end{thebibliography}
\end{document}